\documentclass[a4paper,11pt,oneside,reqno]{amsart}
\usepackage[latin1]{inputenc}
\usepackage[all]{xy}
\usepackage{amssymb, amsmath, amscd, geometry}
\usepackage{graphicx}
\usepackage{setspace}
\input xy
\xyoption{all}

\theoremstyle{plain}

\newtheorem{problem}{Problem}
\newtheorem{definition}{Definition}

\newtheorem{corollary}{Corollary}

\newtheorem{remark}{Remark}
\theoremstyle{plain}
\newtheorem{teore}{Theorem}[section]
\newtheorem{defin}[teore]{Definition}
\newtheorem{lem}[teore]{Lemma}
\newtheorem{coro}[teore]{Corollary}
\newtheorem{propo}[teore]{Proposition}

\newtheorem{claim}{Claim}[teore]
\newtheorem*{claim*}{Claim}
\newtheorem*{thm*}{Theorem}
\newtheorem*{defi*}{Definition}
\theoremstyle{remark}
\newtheorem{ejemplo}{{\sc Example}}
\newtheorem{ejemplos}[teore]{{\sc Examples}}
\newtheorem{notas}[teore]{{\sc Remark}}
\newcommand{\nrm}[1]{\|#1\|}

\newcommand{\prop}{\begin{propo}}
\newcommand{\fprop}{\end{propo}}
\newcommand{\cor}{\begin{coro}}
\newcommand{\fcor}{\end{coro}}

\newcommand{\defi}{\begin{defin}\rm}
\newcommand{\fdefi}{\end{defin}}
\newcommand{\eje}{\begin{ejemplo}}
\newcommand{\feje}{\end{ejemplo}}
\newcommand{\ejes}{\begin{ejemplos}}
\newcommand{\fejes}{\end{ejemplos}}
\newcommand{\lema}{\begin{lem}}
\newcommand{\flema}{\end{lem}}
\newcommand{\teor}{\begin{teore}}
\newcommand{\fteor}{\end{teore}}
\newcommand{\nota}{\begin{notas}\rm}
\newcommand{\fnota}{ \end{notas}}
\newcommand{\clam}{\begin{claim}}
\newcommand{\fclam}{\end{claim}}
\newcommand{\clams}{\begin{claim*}}
\newcommand{\fclams}{\end{claim*}}

\newcommand{\lclam}{\begin{lclaim}}
\newcommand{\flclam}{\end{lclaim}}

\newcommand{\ben}{\begin{enumerate}}
\newcommand{\een}{\end{enumerate}}
\newcommand{\bit}{\begin{itemize}}

\newcommand{\eit}{\end{itemize}}

\newcommand{\mc}[1]{\mathcal{#1}}

\newcommand{\casos}{\begin{itemize}}
\newcommand{\fcasos}{\end{itemize}\setcounter{cs}{1}}

\newcommand{\conj}[2]{ \{ {#1}\,:\,{#2} \} }

\newcommand{\Ga}{\Gamma}

\newcommand{\de}{\delta}
\newcommand{\De}{\Delta}
\newcommand{\la}{\lambda}

\newcommand{\sig}{\sigma}

\newcommand{\vphi}{\varphi}

\newcommand{\vep}{\varepsilon}

\newcommand{\N}{{\mathbb N}}

\newcommand{\prue}{\begin{proof}}
\newcommand{\fprue}{\end{proof}}

\begin{document}
\title{Bases of random unconditional convergence in Banach spaces}
\author[J. Lopez-Abad]{J. Lopez-Abad}
\address{J. Lopez-Abad \\ Instituto de Ciencias Matem\'aticas (ICMAT).  CSIC-UAM-UC3M-UCM. C/ Nicol\'{a}s Cabrera 13-15, Campus Cantoblanco, UAM
28049 Madrid, Spain.}
\address{Instituto de Matem\'atica e Estat\'{\i}stica - IME/USP
Rua do Mat\~{a}o, 1010 - Cidade Universit\'aria, S\~{a}o Paulo - SP, 05508-090, Brasil.}

\email{abad@icmat.es}

\author[P. Tradacete]{P. Tradacete}
\address{P. Tradacete\\Mathematics Department\\ Universidad Carlos III de Madrid\\ 28911, Legan\'es, Madrid, Spain.}
\email{ptradace@math.uc3m.es}

\thanks{Both authors have been partially supported by the Spanish Government Grant MTM2012-31286 and Grupo UCM 910346. The first author acknowledges the support of Fapesp, Grant 2013/24827-1. The second author has also been partially supported by the Spanish Government Grant MTM2010-14946.}

\subjclass[2010]{46B09, 46B15}
\keywords{Unconditional basis, Random unconditional convergence.}

\begin{abstract}
We study random unconditional convergence for a basis in a Banach space. The connections between this notion and classical unconditionality are explored. In particular, we analyze duality relations, reflexivity, uniqueness of these bases and existence of unconditional subsequences.
\end{abstract}
\maketitle

\section{Introduction}

A series $\sum_n x_n$ in a Banach space is \emph{randomly unconditionally convergent} when $\sum_n \vep_n x_n$ converges almost surely on signs $(\vep_n)_n$ (with respect to the Haar probability measure on $\{-1,1\}^\N$). P. Billard, S. Kwapie\'n, A. Pelczy\'nski and Ch. Samuel introduced  in  \cite{BKPS} the notion of random unconditionally convergent (RUC)  coordinate systems $(e_i)_i$ in a Banach space, which have the property that the expansion of every element is randomly unconditionally convergent. Equivalently, a RUC system $(e_i,e_i^*)_i$ in a Banach space satisfies that for a certain constant $K$ and every $x$ in the span of $(e_i)_i$
$$
\sup_n\int_0^1\Big\|\sum_{i=1}^n r_i(t)e_i^*(x)e_i\Big\|dt\leq K\|x\|
$$
where $(r_i)$ is the  sequence of Rademacher function on $[0,1]$. For a RUC Schauder basis $(e_n)$, this is equivalent to
$$
\int_0^1\Big\|\sum_{i=1}^m r_i(t)a_je_i\Big\|dt\leq K\Big\|\sum_{i=1}^ma_ie_i\Big\|
$$
for some constant $K$ independent of the scalars $(a_i)_{i=1}^m$.

It is therefore natural to consider also bases (or more generally, systems) satisfying a converse inequality, i.e.
$$
\|x\|\leq K\int_0^1\Big\|\sum_{i=1}^m r_i(t)e_i^*(x)e_i\Big\|dt.
$$
These will be called random unconditionally divergent (RUD) and satisfy a natural duality relation with RUC systems. These two notions, weaker than that of unconditional basis, are the central objects for our research in this paper.

The search for bases or more general coordinate systems in Banach spaces is a major theme both within the theory and its applications to other areas (signal processing, harmonic analysis...) A basis allows us to represent a space as a space of sequences of scalars via the coordinate expansion of each element. Several interesting properties for bases have been investigated as they provide better, or more efficient ways, to approximate an element in a Banach space. Recall that a a sequence $(x_n)$ of vectors in a Banach space $X$ is called a basis (or Schauder basis) if every $x\in X$ can be written in a unique way as $x=\sum_{n=1}^\infty a_n x_n$, where $(a_n)$ are scalars. It is well-known that this is equivalent to the fact that the projections $P_n(x)=\sum_{i=1}^n a_i x_i$ are uniformly bounded. Among bases, the unconditional ones play a relevant role, as they provide certain extra structure to the space. A basis $(x_n)$ is called  unconditional when the corresponding expansions $\sum_{n=1}^\infty a_n x_n$ converge unconditionally. This is equivalent to the fact that for every choice of signs $\epsilon=(\epsilon_n)$ we have a bounded linear operator $M_{\epsilon}(\sum_{n=1}^\infty a_n x_n)=\sum_{n=1}^\infty \epsilon_na_n x_n$.

There has been considerable interest in finding unconditional basic sequences in Banach spaces.  Since the celebrated paper of W. T. Gowers and B. Maurey \cite{Gowers-Maurey}, we know that not every Banach space contains an unconditional basic sequence. In order to remedy this, weaker versions of unconditionality, such as Elton-unconditionality or Odell-unconditionality, have been considered in the literature \cite{Elton, Odell}. RUC and RUD bases also provide a weakening of unconditionality so several questions arise in a natural way. We will study the relation of these two notions with reflexivity in the spirit of the classical James' theorem \cite{James}, we will investigate the uniqueness of RUC (respectively, RUD) bases in a Banach space, as in \cite{LP}, and several other questions related to unconditionality of subsequences and blocks of a given sequence.

Our approach will begin with some probabilistic observations to illustrate the definition of RUC and RUD bases. We will see how the two notions are related by duality and that they also complement each other, in the sense that a basis which is both RUC and RUD must be unconditional. After this grounding discussion, a list of examples of classical bases which are RUC and/or RUD will be given.

Let us point out a major difference with unconditionality: every block-subsequence of an unconditional basis is also unconditional, whereas this stability may fail for RUC and RUD bases. Actually, every separable Banach space can be linearly embedded in a space with an RUC basis (namely, $C[0,1]$). This follows from \cite{Wojtaszczyk} where it is shown that if a space with a basis contains $c_0$, then it has a RUC basis.

This fact also provides a justification for the hypothesis in our version of James reflexivity theorem in this context (Theorem \ref{James_thm}): Suppose that every block-subsequence of a basis $(x_n)$  is RUD, then $(x_n)$ is shrinking if and only if $X$ does not contain a subspace isomorphic to $\ell_1$. Similarly, if every block subsequence of $(x_n)$  is RUC, then $(x_n)$ is boundedly complete if and only if $X$ does not contain a subspace isomorphic to $c_0$.

Another point worth dwelling on is motivated by the classical theorem of J. Lindenstrauss and A.  Pe\l czynski: the only Banach spaces with a unique, up to equivalence, unconditional basis are $c_0$, $\ell_1$ and $\ell_2$ \cite[2.b]{LT1}. In this respect, it was shown in \cite{BKPS} that all RUC bases of $\ell_1$ are equivalent, and they must be then unconditional; since it is known that conditional RUC bases of $c_0$ and $\ell_2$ exist, $\ell_1$ stands as the only space with this property. However, the situation for the uniqueness of RUD bases is more involved. Of course, the standard argument leaves $c_0$ as the only possible candidate, nevertheless, using a well-known construction of $\mathcal{L}_\infty$ spaces by J.  Bourgain and F. Delbaen \cite{BD}, we will provide a RUD basis of $c_0$ which is not equivalent to the unit vector basis. As a consequence, every Banach space with a RUD basis has another non-equivalent RUD basis.

Let us also recall the first example of a weakly null sequence with no unconditional subsequences due to B.  Maurey and H. P. Rosenthal \cite{MR}. It can be seen that this construction also produces an example of a weakly null sequence with no RUD subsequence (Theorem \ref{Maurey-Rosenthal}). Based on this we can provide a weakly null RUC basis without unconditional subsequences (Theorem \ref{MR-RUC}). Using equi-distributed sequences of signs, a modification of Maurey-Rosenthal construction can be given to build a RUD basis without unconditional subsequences (Theorem \ref{MR-RUD}). Moreover, this example also shows that normalized blocks of a RUD basis need not be RUD. Incidentally, the construction of a weakly null sequence in the space $L_1$ without unconditional subsequences given in \cite{JMS} by W. Johnson, B. Maurey and G. Schechtman, can also be taken to be RUD. In fact, it will be shown that on r.i. spaces which are separated from $L_\infty$ (in the sense that the upper Boyd index is finite) every weakly null sequence has an RUD subsequence (Theorem \ref{ri RUD}).

The research on RUC and RUD bases gives rise to a number of natural  questions concerning unconditionality in Banach spaces. Among them, the fundamental question of whether every Banach space contains an RUD or an RUC basic sequence remains open.

Throughout the paper we follow standard terminology concerning Banach spaces as in the monographs \cite{LT1, LT2}, and for questions related to probability the reader is referred to \cite{Ledoux-Talagrand} and \cite{Loeve}.

\section{RUC and RUD bases}

\begin{definition}
A series   $\sum_n x_n$ in a Banach  space is \emph{randomly unconditionally convergent} when $\sum_n \vep_n x_n$ converges
almost surely on signs $(\vep_n)_n$ with respect to the Haar probability measure on $\{-1,1\}^\N$, or,
equivalently, when the series $\sum_n r_n(t)x_n$ converges almost surely with respect to the Lebesgue measure
on $[0,1]$, where  $(r_n(t))_n$  is the \emph{Rademacher} sequence in $[0,1]$.
\end{definition}
Since the convergence  does not
depend on finitely many changes, it follows from the corresponding 0-1 law that  either  $\sum_n \vep_n x_n$ converges a.s. or   $\sum_n \vep_n x_n$ diverges a.s.   (see \cite[pp 7]{Ka} for more details).
Recall the following fact, know as the {\it contraction principle}.
\prop
Suppose that $\sum_n r_n(t)x_n$ converges a.s. Then for every sequence $(a_n)_n$, $\sup_n |a_n|\le 1$, one
has that $\sum_n a_n r_n(t)x_n$ also converges a.s.
 \qed
 \fprop
Consequently, the sequence $(r_n x_n)_n$ in the Bochner space $L_1([0,1],X)$ is a 1-unconditional basic sequence.

We recall the corresponding expected value
\begin{equation}\label{oiioe4joigfofg}
\mathbb{E} \Big(\Big\|\sum_{n=1}^m \epsilon_n x_n \Big\| \Big)=\frac{1}{2^m}\sum_{(\epsilon_n)\in\{-1,+1\}^m}\Big\|\sum_{n=1}^m \epsilon_n x_n \Big\|=\int_0^1\Big\|\sum_{n=1}^m r_n(t)x_n\Big\|_Xdt.
\end{equation}

It is shown by J. P. Kahane \cite[Theorem 4]{Ka} that if $\sum_n r_n(t) x_n$ converges a.s., then $\mathbb{E} \nrm{\sum_{n}r_n(t) x_n}<\infty$, i.e.
the $X$-vector valued function $\sum_n r_n(t) x_n$ belongs to the \emph{Bochner}  space $L_1([0,1],X)$. The converse is also true when $(x_n)_n$ is basic.

\prop
Suppose that $(x_n)_n$ is basic and  $(a_n)_n$ is a sequence such that the series $\sum_n a_n r_n(t) x_n $
Bochner-converges. Then $\sum_n r_n(t) a_n x_n$ converges almost surely.
 \fprop
\begin{proof}
 Suppose that $(s_n(t))_{n}$ converges to an $X$-valued Bochner measurable function $f$, where $s_n(t):=\sum_{i=1}^n a_i r_i(t) x_i$ for every $n$. This means that $\int_0^1 \nrm{s_n(t)-f(t)}_X\to_n 0$.
 Hence, $\nrm{s_n(t)-f(t)}_X\to_n 0$ in probability.  It follows that there is a subsequence $\nrm{s_{n_k}(t)-f(t)}_X\to_k 0$ almost surely.
  In particular, $(s_{n_k}(t))_k$ is a Cauchy sequence almost surely. We prove that $(s_{n}(t))_n$ is in fact a Cauchy
  sequence almost surely: Let
 $$A:=\conj{t\in [0,1]}{(s_{n_k}(t))_k\text{ is a Cauchy sequence}}.$$
 By hypothesis, $\la(A)=1$. Then $(s_n(t))_n$ is Cauchy for every $t\in A$: Let $C$ be the basic constant
  of $(x_n)_n$, and given $\vep>0$, let $k_\vep$ be such that $\nrm{s_{n_k}(t)-s_{n_l}}\le \vep/(2C)$ for every
  $k,l \ge k_\vep$. Then, using that $(x_n)_n$ is $C$-basic, if $n_{k_\vep}\le m\le n$, it follows that
 \begin{align*}
 \nrm{s_m(t)-s_n(t)}\le 2C\nrm{s_{n_{k_\vep}}(t)- s_{n_l}(t)}\le \vep,
 \end{align*}
  where $l$ is such that $n_l\ge n$. We have just proved that $\sum_n r_n(t) a_n x_n$ converges almost surely to $f(t)$.
\end{proof}

A complete account on series of the form $\sum_n \epsilon_n x_n$, also referred to as Rademacher averages, can be found in \cite[Chapter 4]{Ledoux-Talagrand}.

\subsection{Definition and basic properties}

\begin{definition}
A  basic sequence $(x_n)_n$ in a Banach space $X$ is of  \emph{Random Unconditional Convergence} (a RUC basis in short)
when every convergent series $\sum_n a_n x_n$ is randomly unconditionally convergent.

A basic sequence $(x_n)_n$ of $X$ is called of \emph{Random Unconditional Divergence} (RUD basis in short) when
whenever a series $\sum_n a_n x_n$ is randomly unconditional, the series $\sum_n a_n x_n$ is convergent, or equivalently,
the only randomly unconditional series $\sum_n a_n x_n$ are the unconditional ones.
\end{definition}

It is clear that the definition extends to biorthogonal systems in a natural way. The terminology is justified by the 0-1 law implying that  $(x_n)_n$ is RUD if and only if for every
divergent series $\sum_n a_n x_n$ the signed series $\sum_n \vep_n a_n x_n$ diverges almost surely. RUC
bases are those with the maximal number of random unconditionally convergent series, while RUD bases are those with the minimal number of them, only the unconditional ones.

\prop
A basic sequence  is unconditional if and only if it is RUC and RUD.
\fprop
\prue
Suppose that $(x_n)_n$ is a RUC and RUD basic sequence, suppose that $\sum_n a_n x_n$ converges and let $(\sig_n)_n$ be a sequence of signs. We have to prove that $\sum_n \sig_n a_n x_n$ also converges.  Suppose otherwise that $\sum_n \sig_n a_n x_n$ diverges. Since $(x_n)_n$ is RUC, it follows that $\sum_n \vep_n\sig_n a_n x_n$ diverges a.s. in $(\vep_n)_n$, or equivalently, $\sum_n \vep_n a_n x_n$ diverges a.s. Since $(x_n)_n$ is RUC, it follows that $\sum_n a_n x_n$ diverges, a contradiction.
\fprue

RUC sequences
were introduced by  P. Billard, S. Kwapie\'n, A. Pelczy\'nski and Ch. Samuel in \cite{BKPS}, where they prove
the following  quantitative characterization for RUC biorthogonal systems.
 \prop\label{ljejrejriedfd}
For a basic sequence $(x_n)_n$ in $X$ the following are equivalent.
\begin{enumerate}
\item[(a)] $(x_n)_n$ is RUC.
\item[(b)] There is a constant $C$ such that for every $n\in \N$ and every sequence of scalars $(a_i)_{i=1}^n$ one has that
\begin{equation}
\label{ljjgijfgf}  \mathbb E \nrm{\sum_{i=1}^n \vep_i a_i x_i}\le C\nrm{\sum_{i=1}^n a_i x_i}.
\end{equation}
\end{enumerate}
 \fprop
In a similar way,  we have the following.
\prop
Let $(x_n)_n$ be a basic sequence of $X$. The following are equivalent.
\begin{enumerate}
\item[(a)] $(x_n)_n$ is RUD.
\item[(b)] There is a constant $C$ such that for every $n\in \N$ and every sequence of scalars $(a_i)_{i=1}^n$ one has that
\begin{equation}\label{knlknlknjr5565f}
\nrm{\sum_{i=1}^n a_i x_i}\le C\mathbb E \nrm{\sum_{i=1}^n \vep_i a_i x_i}.
\end{equation}
\end{enumerate}
 \fprop
\begin{proof}
Suppose that $(x_n)_n$ is RUD. This implies that $\sum_n a_n x_n$ converges whenever $\sum_n \vep_n a_n x_n$
a.s. converges. Let $Y$ be the closed subspace of the Bochner space $L_1([0,1],X)$ spanned by
$(r_n(t)x_n)_{n\in \N}$. Since  $(r_n(t)x_n)_n$ is a 1-unconditional basis of $Y$,  for each $n\in \N$, the
linear operator $S_n:Y\to X$  defined by $S_n(\sum_i a_i r_i(t)x_i)=\sum_{i=1}^n a_ix_i$ is well defined and bounded. Now, for
a fixed $y=\sum_i a_i r_i(t)x_i\in Y$ we know by hypothesis that $\sum_i a_i x_i$ converges; since $(x_i)_i$ is a basic sequence, with basic constant $K$, it follows that
\begin{equation}
\nrm{S_n(y)}=\nrm{\sum_{i=1}^n a_i x_i}\le K
\nrm{\sum_{i=1}^\infty a_i x_i}
\end{equation}
for every $n$. Hence, by the Banach-Steinhaus
principle, it follows that that $C:=\sup_n \nrm{S_n}<\infty$, that is,
\begin{equation}
\nrm{\sum_{i=1}^n a_i x_i}\le C\nrm{\sum_{i=1}^\infty r_i(t) a_i x_i}_{L_1([0,1],X)}.
\end{equation}
For a fixed $n$, if we replace $(a_i)_i$ by $(b_i)_i$ where $b_i=a_i$ for $i\le n$ and $b_i=0$ otherwise, we
obtain the inequality in \eqref{knlknlknjr5565f}.

Suppose now that \eqref{knlknlknjr5565f} holds for every $n$ and every $(a_i)_{i=1}^n$.   Suppose that
$\sum_{n}a_n x_n$ diverges. If $\sum_{n}\vep_n a_n x_n$ does not diverge  a.s., by the 0-1 Law, it converges
a.s.  Hence, by Kahane's result, it follows that $\mathbb E_{(\vep_i)}\nrm{\sum_i \vep_i a_i x_i}<\infty$, or equivalently,
$\sum_i a_i r_i(t)x_i $ converges in $L_1([0,1],X)$. It follows that $(\sum_{i=1}^n a_i r_i(t)x_i)_{n}$ is a
Cauchy sequence. Now, the inequality in \eqref{knlknlknjr5565f} implies that $(\sum_{i=1}^na_i x_i)_{n}$ is also
Cauchy, a contradiction.
\end{proof}

\begin{remark}{\rm
\begin{enumerate}
\item[(a)] Sequences that satisfy the inequality in \eqref{ljjgijfgf} are obviously biorthogonal, and in fact the characterization in Proposition \ref{ljejrejriedfd} is still valid for biorthogonal sequences.
\item[(b)] On the other hand, an arbitrary  semi normalized sequence satisfying the inequality in \eqref{knlknlknjr5565f} must have basic subsequences: By applying Rosenthal's $\ell_1$ Theorem to $(r_n(t) x_n)_n$, there are two cases to consider: suppose first that there is a subsequence $(r_n(t) x_n)_{n\in M}$ equivalent to the unit basis of $\ell_1$. It follows then that there is a subsequence $(x_n)_{n\in N}$ equivalent to the unit basis of $\ell_1$ (see Proposition \ref{iuuiuiere}), hence basic. Otherwise, there is a weakly-Cauchy subsequence $(r_n(t) x_n)_{n\in M}$. Since this sequence is 1-unconditional, it must be weakly-null: otherwise, $(r_n(t) x_n)_{n\in M}$ is not weakly-convergent, hence it has a basic subsequence $(r_n(t) x_n)_{n\in N}$ which dominates the summing basis of $c_0$; since $(r_n(t)x_n)_{n\in N}$ is unconditional and bounded, it will be equivalent to the unit basis of $\ell_1$, so it cannot be weakly-Cauchy. Now from the fact that $(r_n(t)x_n)_{n\in M}$ is weakly-null and the inequality in  \eqref{knlknlknjr5565f} it follows that  $(x_n)_{n\in M}$ is also weakly-null, and consequently it has a further basic subsequence.
\item[(c)]There is a significant difference if almost everywhere convergence of the series $\sum_{i=1}^n \epsilon_i a_i x_i$ is replaced by quasi-everywhere convergence, that is when the set of signs for which the series converges contains a dense $G_\delta$. This last condition is equivalent to the unconditionality of the basic sequence $(x_i)_i$, as it has been proved by P. Lefevre in \cite{Lefevre}.

\end{enumerate}}
\end{remark}

\defi
A RUC (RUD) basic sequence $(x_n)_n$ is $C$-RUC ($C$-RUD) when the inequality  in \eqref{ljjgijfgf}   (resp. \eqref{knlknlknjr5565f}) holds.   The corresponding RUC and RUD constants are defined naturally as
\begin{align*}
\mathrm{RUC}((x_n)_n):= &\inf\{C>0: \|\sum_{i=1}^n a_i e_i\|\ge \frac 1 C \mathbb{E} \Big(\Big\|\sum_{i=1}^n \epsilon_i a_i x_i \Big\| \Big)\},\\
\mathrm{RUD}((x_n)_n)):= &\inf\{C>0: \|\sum_{i=1}^n a_i x_i\|\leq C\mathbb{E} \Big(\Big\|\sum_{i=1}^n \epsilon_i a_i x_i \Big\| \Big)\},
\end{align*}
where the infimums are taken over all finite sequences $(a_i)_{i=1}^n$ of scalars.
\fdefi

It is also clear from the definition is that if $(e_n)$ is RUC (RUD), then for any choice of scalars
$\lambda_n$, the sequence $(\lambda_n e_n)$ is also RUC (resp. RUD) (with the same constant).

Since we always have the inequalities
\begin{equation}
\min_{\tau_n=\pm1}\Big\|\sum_{n=1}^m \tau_n a_n x_n \Big\| \leq\mathbb{E} \Big(\Big\|\sum_{n=1}^m \epsilon_n a_n x_n \Big\| \Big)\leq \max_{\tau_n=\pm1}\Big\|\sum_{n=1}^m \tau_n a_n x_n \Big\|
\end{equation}
it follows that the RUC and RUD constants, if they exist, are at least 1. In fact, we have the following simple characterizations.

\prop
\label{characterization_ruc}
Let  $(x_n)_n$ be a basic sequence. The following are equivalent:
\begin{enumerate}
\item $(x_n)_n$ is $C$-RUC.
\item For any sequence of scalars $(a_n)_{n=1}^m$ we have
$$
\min_{\tau_n=\pm1}\Big\|\sum_{n=1}^m \tau_n a_n x_n \Big\| \leq \mathbb{E} \Big(\Big\|\sum_{n=1}^m \epsilon_n a_n x_n \Big\| \Big)\leq   C
\min_{\tau_n=\pm1}\Big\|\sum_{n=1}^m \tau_n a_n x_n \Big\|.
$$
\end{enumerate}
 Consequently,  $(x_n)_n$ is 1-RUC if and only if $(x_n)_n$ is 1-unconditional.
 \fprop

\prop
\label{characterization_sub}
Let  $(x_n)_n$ be a basic sequence. The following are equivalent:
\begin{enumerate}
\item $(x_n)_n$ is $C$-RUD.
\item For any sequence of scalars $(a_n)_{n=1}^m$ we have
$$
\mathbb{E} \Big(\Big\|\sum_{n=1}^m \epsilon_n a_n x_n \Big\| \Big)\leq \max_{\tau_n=\pm1}\Big\|\sum_{n=1}^m \tau_n a_n x_n \Big\|\leq C \mathbb{E} \Big(\Big\|\sum_{n=1}^m \epsilon_n a_n x_n \Big\| \Big).
$$
\end{enumerate}
 Consequently,  $(x_n)_n$ is 1-RUD if and only if $(x_n)_n$ is 1-unconditional.
\fprop
In the case of RUC basic sequences, we can always
renorm the space to get RUC-constant as close to one as desired. We do not know if the same is true for RUD basic sequences.
\prop
Let $(x_n)$ be a RUC basic sequence in  $X$. For every $\delta>0$ there is an equivalent norm in $X$ such
that $(x_n)$ is $(1+\delta)$-RUC, although there are examples for every $\de>0$ of $(1+\de)$-RUD sequences without unconditional subsequences (see Theorem \ref{MR-RUD}).
\fprop

\begin{proof}
Without loss of generality, we may assume that $(x_n)_n$ is a basis of $X$. Let $\|\cdot\|$ denote the norm in $X$ such that for some $C>1$
$$
\mathbb{E}\Big\|\sum_{n}a_n\vep_n x_n\Big\|\leq C \Big\|\sum_{n}a_n x_n\Big\|.
$$
Given $\delta>0$, let us define a new norm
$$
\Big\|\sum_{n}a_n x_n\Big\|_{\delta}=\mathbb{E}\Big\|\sum_{n}a_n\vep_n x_n\Big\|+\delta\Big\|\sum_{n}a_n x_n\Big\|.
$$
It is clear that
$$
\delta\|\cdot\|\leq\|\cdot\|_\delta\leq(C+\delta)\|\cdot\|,
$$
while we have
\begin{align*}
\mathbb{E}\Big\|\sum_{n}a_n\vep_n x_n\Big\|_\delta=  &\mathbb{E}\Big(\mathbb{E}\Big\|\sum_{n}a_n\vep_n x_n\Big\|+\delta\Big\|\sum_{n}a_n\vep_n x_n\Big\|\Big)=(1+\delta)\mathbb{E}\Big\|\sum_{n}a_n\vep_n x_n\Big\|\leq \\
\leq & (1+\delta)\Big\|\sum_{n}a_n x_n\Big\|_\delta.
\end{align*}
\end{proof}

The signs-average given above is equivalent (i.e. up to a universal constant) to the following
subsets-average.
$$
\mathbb{E}_0 \Big(\Big\|\sum_{n=1}^m \theta_n x_n \Big\|\Big)=\frac{1}{2^m}\sum_{(\theta_n)\in\{0,1\}^m}\Big\|\sum_{n=1}^m \theta_n x_n \Big\|=\frac{1}{2^m}\sum_{A\subset\{1,\ldots,m\}}\Big\|\sum_{n\in A}x_n \Big\|.
$$
More precisely,
$$
\mathbb{E}_0 \Big(\Big\|\sum_{n=1}^m \theta_n x_n \Big\|\Big) \leq \mathbb{E} \Big(\Big\|\sum_{n=1}^m \epsilon_n x_n \Big\| \Big) \leq 2 \mathbb{E}_0 \Big(\Big\|\sum_{n=1}^m \theta_n x_n \Big\|\Big).
$$

It is also natural to consider random versions of symmetric bases. For instance, if $\Pi_n$ denotes the group of permutations of $\{1,\dots,n\}$, and we consider a finite basis $(x_i)_{i=1}^n$ and scalars $(a_i)_{i=1}^n$, we can define
$$
\mathbb E_\pi \nrm{\sum_{i=1}^n a_{\pi(i)} x_{i}}:=\frac{1}{n!}\sum_{\pi\in \Pi_n}\nrm{\sum_{i=1}^n a_{\pi(i)} x_{i}}.
$$
Hence, we say that a basis $(x_i)$ is of Random Symmetric Convergence (RSC in short) with constant $C$ when for every $n\in\mathbb N$ and scalars $(a_i)_{i=1}^n$
\begin{equation}
\mathbb E_\pi \nrm{\sum_{i=1}^n a_{\pi(i)} x_{i}}\le C \nrm{\sum_{i=1}^n a_i x_{i}}.
\end{equation}
Similarly, $(x_i)$ is of Random Symmetric Divergence
(RSD in short) with constant $C$ when
\begin{equation}
\nrm{\sum_{i=1}^n a_i x_{i}} \le C \mathbb E_\pi \nrm{\sum_{i=1}^n a_{\pi(i)} x_{i}}
\end{equation}
for every choice of $n$ and scalars $(a_i)_{i=1}^n$. The research of these notions will be carried out elsewhere.

Recall that given an integer $k$ and a property $\mc P$ of sequences in a given space $X$ we say that a sequence $(x_n)_n$ has the $k-$skipping property $\mc P$ when every subsequence $(x_{n_i})_{i}$ of $(x_n)_n$ has the property $\mc P$ provided that $n_{i+1}-n_{i}\ge k$.

\prop\label{hohriogihohgfhg}
Let $(x_n)_{n\in I}$ be a basic sequence in $X$, $I$ finite or infinite.
\begin{enumerate}
\item[(a)]  If $(x_n)_n$ is $k$-skipping RUD for some $k\in \N$, then it is RUD.
In fact, suppose that $I=P_1\cup \cdots \cup P_k$ is a partition of $I$ such that each subsequence $(x_n)_{n\in P_i}$ is RUD with constant $C_i$,
$i=1,\dots,k$, then $(x_n)_{n\in I}$ is RUD with constant $\le \sum_{i=1}^kC_i$.
\item[(b)]    Suppose that $(x_n)_n$ is a RUC basis of $X$. Then every unconditional subsequence of it generates a complemented
subspace of $X$.
\end{enumerate}
\fprop
\begin{proof}
(a):  Suppose that $\sum_n r_n(t) a_n x_n$ converges a.s. It follows from the contraction principle that each $\sum_{n\in P_i}r_n(t) a_n x_n$, $i=1,\dots,n$, is also convergent a.s. Hence each series $\sum_{n\in P_i} a_n x_n$ converges, $i=1,\dots,n$, and consequently also $\sum_n a_n x_n$ converges. As for the constants:       Fix $n$ and scalars $(a_i)_{i=1}^n$. Then
\begin{align*}
\nrm{\sum_{i=1}^n a_i x_i}\le & \sum_{j=1}^k\nrm{\sum_{i\in   P_j\cap \{1,\dots, n\} } a_i x_i   } \le  \sum_{j=1}^k C_j\mathbb E_{\vep}
\nrm{\sum_{i\in   P_j\cap \{1,\dots, n\} } \vep_i a_i x_i }\le \\
\le &\sum_{j=1}^k C_j  \mathbb E_{\vep}
\nrm{\sum_{i=1  }^n \vep_i a_i x_i}.
\end{align*}
(b): Suppose that $(x_n)_{n\in M}$ is unconditional. We claim that the boolean projection $\sum_{n}a_n x_n
\mapsto \sum_{n\in M }a_n x_n$ is bounded:
\begin{align*}
\nrm{\sum_{n\in M}a_n x_n}\approx \mathbb E_\vep \nrm{\sum_{n\in M}\vep_n a_n x_n} \le  \mathbb E_\vep \nrm{\sum_{n}\vep_n a_n x_n} \lesssim \nrm{\sum_{n} a_n x_n} .
\end{align*}
\end{proof}

\begin{corollary}\label{ufdd}
Suppose $X$ is a Banach space with an unconditional f.d.d. $(F_n)_n$ such that
$$
\sup_n \dim F_n<\infty.
$$
Then $X$ has a RUD basis.
\end{corollary}
\begin{proof}
Choose for each $n$ a basis $(x_i^{(n)})_{i<k_n}$, $k_n:=\dim F_n$ with basic constant $\le C$, independent of $n$. Then $(x_i^{(n)})_{i<k_n,n\in \N}$ ordered naturally $(x_j)_j$ is a Schauder basis of $X$, and  it is $k$-skipping RUD.
\end{proof}

Let us establish now some duality relation between RUC and RUD bases. Recall that a functional $x^*\in X^*$ and a function $f\in L_2(0,1)$ always define an element in
$L_2((0,1),X)^*$ as follows: for any $g\in L_2((0,1),X)$
$$
f\otimes x^* (g):=\int_0^1 \langle x^*,g(t)\rangle f(t) dt.
$$

\prop\label{duality} Let $(x_n)_n$ be a basis of $X$.
\begin{enumerate}
\item If $(x_n)$ is $C$-RUC then every biorthogonal sequence $(x_n^*)$ is $2C$-RUD.
\item If $(x_n^*)_n$ is $C$-RUC, then $(x_n)_n$ is $C\cdot D$-RUD, where $D$ is the basic constant of $(x_n)_n$.
\end{enumerate}
\fprop

\begin{proof}
Suppose that $(x_n)$ is a RUC  basis of the space $X$ with RUC constant $C$, and let  $(x_n^*)\subset X^*$ be its  sequence of biorthogonal functionals.

Now, fix $\sum_{i=1}^n b_i x_i^*\in X^*$, and let $x=\sum_{i=1}^n a_i x_i$  be such that $\|x\|=1$ and
$$
\sum_{i=1}^n a_i b_i=\langle x, \sum_{i=1}^n b_i x_i^*\rangle=\|\sum_{i=1}^n b_ix_i^*\|.
$$
Since $(x_n)_n$ is RUC with RUC constant $C$, it follows from Khintchine-Kahane that
$$ \nrm{\sum_{ i=1}^n a_i r_i(t) x_i^*}_{L_2([0,1],X)} \le \sqrt{2}\nrm{\sum_{ i=1}^n a_i r_i(t) x_i^*}_{L_1([0,1],X)}\le \sqrt{2}C\nrm{x}=\sqrt{2}C.$$
Hence,
\begin{align*}
\nrm{\sum_{i=1}^n b_i r_i(t) x_i^*}_{L_1([0,1],X^*)}   \ge & \frac1{\sqrt{2}}
\nrm{\sum_{i=1}^n b_i r_i(t) x_i^*}_{L_2([0,1],X^*)} \ge \\
\ge & \frac1{2C}\langle \sum_{i=1}^n a_i r_i(t) x_i, \sum_{i=1}^n b_i r_i(t) x_i^* \rangle= \frac1{2C}  \sum_{i=1}^n a_i b_i = \\
= & \frac1{2C} \|\sum_{i=1}^n b_ix_i^*\|
\end{align*}
Hence, $(x_n^*)$ is RUD with basic constant $\le 2C$.

The proof of (2) is done similarly now observing that the unit sphere of $\langle x_n^*\rangle_n$ is $1/D$-norming, where $D$ is the basic constant of $(x_n)_n$.
\end{proof}
The corresponding duality result for RUD bases is not true in general (see Example \ref{summing}).
We will give now a version of James theorem characterizing shrinking and boundedly complete unconditional
basis in terms of subspaces isomorphic to $\ell_1$ and $c_0$.
\teor\label{James_thm} Let $(x_n)_n$ be a basis of a Banach space $X$.
\begin{enumerate}
\item Suppose that every block subsequence of $(x_n)$  is RUD. Then   $(x_n)$ is shrinking if and only if $X$ does not contain a subspace isomorphic to $\ell_1$.
\item Suppose that every block subsequence of $(x_n)$  is RUC. Then   $(x_n)$ is boundedly complete if and only if $X$ does not contain a subspace isomorphic to $c_0$.
\end{enumerate}
\fteor

\begin{proof}
$(1)$ Clearly, if $X$ contains a subspace isomorphic to $\ell_1$, then $\ell_\infty$ is a quotient of $X^*$.
Thus, $X^*$ is non-separable, and $(x_n)$ cannot be shrinking. Conversely, suppose that $(x_n)$ fails to be
shrinking. This means that for some $\varepsilon>0$ and $x^*\in X^*$ with $\|x^*\|=1$ we can find blocks
$(u_j)$ of the basis $(x_n)$ such that $x^*(u_j)\geq\varepsilon$ for every $j\in\mathbb{N}$. Since $(u_j)$ is
RUD, given scalars $(a_j)_{j=1}^m$ we have
\begin{align*}
\mathbb{E} \Big(\Big\|\sum_{j=1}^m \epsilon_j a_j u_j \Big\| \Big)= & \mathbb{E} \Big(\Big\|\sum_{j=1}^m \epsilon_j |a_j| u_j \Big\| \Big)\geq C \Big\|\sum_{j=1}^m |a_j| u_j\Big\|\geq Cx^*\big(\sum_{j=1}^m |a_j| u_j\big)\geq \\
\geq & C\varepsilon\sum_{j=1}^m |a_j|.
\end{align*}
Therefore, we have the equivalence
$$
\mathbb{E} \Big(\Big\|\sum_{j=1}^m \epsilon_j a_j u_j \Big\| \Big)\approx\sum_{j=1}^m |a_j|
$$
which, by a Result of Bourgain in \cite{Bo2} (see also Proposition \ref{iuuiuiere} below), implies that there
is a further subsequence $(u_{j_k})$ equivalent to the unit basis of $\ell_1$.

$(2)$: If $X$ has a subspace isomorphic to $c_0$, then it is easy to see that the basis $(x_n)$ cannot be
boundedly complete. Conversely, let us assume that $(x_n)$ is not boundedly complete. Thus, there exist
scalars $(\lambda_n)$ such that
$$
\sup_m\Big\|\sum_{n=1}^m \lambda_nx_n\Big\|\leq1,
$$
but the series
$$
\sum_{n=1}^\infty \lambda_nx_n
$$
does not converge. This means that for some increasing sequence of natural numbers $(p_k)_{k\in\mathbb{N}}$
and some $\varepsilon>0$ we have
$$
u_k=\sum_{j=p_{2k}+1}^{p_{2k+1}}\lambda_j x_j,
$$
with $\|u_k\|\geq\varepsilon$, for $k\in\mathbb{N}$. Hence, since $(u_k)$ is a block sequence, then it is
RUC, and we have
$$
\sup_m\int_0^1\Big\|\sum_{i=1}^m r_i(t)u_i\Big\|dt=\sup_m \mathbb{E} \Big(\Big\|\sum_{i=1}^m \epsilon_i u_i \Big\| \Big)\leq C\sup_m \Big\|\sum_{i=1}^m u_i\Big\|\leq C<\infty.
$$
By a result of Kwapien in \cite{Kw}  (see   Theorem \ref{oiio3rere} for more details), $(u_k)$ has a
subsequence equivalent to the unit basis of $c_0$.
\end{proof}

\begin{problem}
Suppose that $(x_n)_n$ is a basis of $X$ such that every block-subsequence of $(x_n)_n$ is RUC (equiv. RUD).
Is $(x_n)_n$ unconditional? More generally, does there exist an unconditional block-subsequence of $(x_n)_n$?
\end{problem}

We will see in Section \ref{Sec_RUD_Banach} that there exist conditional basis (namely, the Haar basis in $L_1$) such that every block subsequence  is RUD.

\subsection{Examples}

We will present next a list of examples of classical bases in Banach spaces, illustrating the notions of RUC and RUD bases. Let us begin with an example of a basis without RUC nor RUD subsequences.

\eje\label{summing}
The summing basis $(s_n)$ in $c_0$  does not have RUD or RUC subsequences, but its biorthogonal sequence in $\ell_1$ is RUD.
\feje

\begin{proof}
Recall that the $n^\mathrm{th}$ term $s_n$ of the summing basis is the sequence
$$ s_n:=\sum_{i=1}^n u_i=(\overset{(n)}{\overbrace{1,\dots,1}},0,0,\dots),$$
where $(u_n)_n$ is the unit basis of $c_0$. It follows that for any  finite subset $s$ of $\mathbb N$ and any
sequence of scalars $(a_i)_{i\in s}$ it holds that
$$
\Big\|\sum_{i\in s} a_i s_i\Big\|=\max_{m\in s}\Big|\sum_{i \in s,\, i\ge m} a_i\Big|.
$$
We claim that
$$
\mathbb{E}_\vep \Big(\Big\|\sum_{i\in s} \epsilon_i a_i s_i \Big\| \Big)\approx\Big(\sum_{i\in s} a_i^2\Big)^{\frac12}.
$$
Indeed, we have that
$$
\mathbb{E}_\vep \Big(\Big\|\sum_{i\in s} \epsilon_i a_i s_i \Big\| \Big)=
\int_0^1\max_{m\in s}\Big|\sum_{i\in s,\, i\ge m} a_ir_i(t)\Big|dt.
$$
Now, Levy's inequality (cf. \cite[2.3]{Ledoux-Talagrand}, \cite[p. 247]{Loeve}) yields
$$
\mu\{t\in[0,1]:\max_{m\in s}\Big|\sum_{i\in s,\, i\ge m}  a_ir_i(t)\Big|\geq s\} \leq 2\, \mu\{t\in[0,1]:\Big|\sum_{i\in s} a_ir_i(t)\Big|\geq s\}.
$$
Hence, this fact together with Khintchine's inequality give that
\begin{align*}
\frac1{\sqrt{2}}\Big(\sum_{i\in s} a_i^2\Big)^{\frac12} \le & \int_0^1 \Big|\sum_{i\in s} a_ir_i(t)\Big|dt    \le
  \int_0^1\max_{ m\in s}\Big|\sum_{i\in s,\, i\ge m} a_ir_i(t)\Big|dt\leq \\
  \leq &
 2\int_0^1\Big|\sum_{i\in s} a_ir_i(t)\Big|dt\leq 2\Big(\sum_{i\in s}a_i^2\Big)^{\frac12}.
\end{align*}
In particular, there is no constant $K\ge 1$ such that for every  finite subset $s$ of a given infinite
$N\subseteq \mathbb N$ we could have
$$
\sharp s=\Big\|\sum_{i\in s} s_i\Big\|\leq K \mathbb{E}_\vep \Big(\Big\|\sum_{i\in s }^m \epsilon_i s_i \Big\| \Big)\leq 2K\sqrt{\sharp s},
$$
and there is no constant $K\ge 1$ such that for every  $n_1<\dots <n_k$ in $ N$,
$$
1=\Big\|\sum_{i=1}^k (-1)^i s_{n_i}\Big\|\ge \frac1K \mathbb{E}_\vep \Big(\Big\|\sum_{i=1 }^k \epsilon_i (-1)^i s_i \Big\| \Big)\ge
 \frac{1}{K \sqrt{2}}
\sqrt{k}.
$$

The biorthogonal sequence $(s_n^*)_n$ in $\ell_1$ to $(s_n)_n$ is RUD: To see this, notice that
$s_n^*=u_n-u_{n+1}$ for every $n$, where $(u_n)_n$ is the unit basis of $\ell_1$. Hence, for every sequence
of scalars $(a_i)_{i=1}^n$ one has that $\nrm{\sum_{i=1}^n a_i
s_i^*}_1=|a_1|+\sum_{i=1}^{n-1}|a_i-a_{i+1}|+|a_n|$. Consequently,
\begin{align*}
\mathbb E_\vep \nrm{\sum_{i=1}^n a_i \vep_i s_i^*}_1= |a_1|+|a_n|+\sum_{i=1}^{n-1}\frac{1}{2}( |  a_i + a_{i+1}|+|a_i-a_{i+1}|)\ge \sum_{i=1}^n |a_i|.
\end{align*}
Since $\nrm{s_n^*}=2$ for every $n$, it follows that
\begin{equation}
\nrm{\sum_{i=1}^n a_i s_i^*}\le 2\mathbb E_\vep \nrm{\sum_{i=1}^n a_i \vep_i s_i^* }.
\end{equation}
\end{proof}

Note that proving the conditionality of $(s_n)$ is considerably simpler than showing that it is not RUC nor RUD, for which some probability technology is employed. In this case, Levy's inequality makes the trick, but for slightly more general situations other estimates like H\`ajek-R\'enyi inequality can be helpful \cite{HR}: If $X_1,\ldots, X_n$ are independent centered random variables, $S_k=\sum_{i=1}^k X_i$, and $c_1\geq c_2\geq\ldots\geq c_n\geq0$, then we have
$$
\mu\{\max_{1\leq k\leq n}c_k|S_k|\geq\varepsilon\}\leq\varepsilon^{-2}\int c_n^2S_n^2+\sum_{i=1}^{n-1}(c_i^2-c_{i+1}^2)S_i^2d\mu.
$$

Let us provide now an example of a RUD basis which is not unconditional. Recall first that James space $J$
\cite{James} is the completion of the space of eventually null sequences $c_{00}$ under the norm
$$
\|(a_n)_n\|_J=\sup\{\big(\sum_{k=1}^m(a_{p_k}-a_{p_{k+1}})^2\big)^{\frac12}:p_1<p_2<\cdots<p_{m+1}\}.
$$

\eje\label{james}
The unit vector basis $(u_n)$ of James space $J$ is RUD. In fact, it is a conditional RUD basis whose expected value is the unit basis of $\ell_2$.
\feje

\begin{proof}
Let us consider an arbitrary sequence of scalars $(a_i)_{i=1}^m$ and let $p_1<p_2<\cdots<p_n$ be such that
$$
\|\sum_{i=1}^ma_iu_i\|_J=\big(\sum_{j=1}^n(a_{p_j}-a_{p_{j+1}})^2\big)^{\frac12}.
$$
It follows that
\begin{equation}
\nrm{\sum_i a_i u_i}_J^2= \sum_{j=1}^n(a_{p_j}-a_{p_{j+1}})^2 \le \sum_{j=1}^n(a_{p_j}^2+a_{p_{j+1}}^2)\le \sum_{i} a_i^2
\end{equation}
 Hence,
 \begin{equation}\label{dfsfdsdfddd333ds}
\nrm{\sum_i a_i u_i}_J \le \nrm{\sum_i a_i u_i}_{\ell_2}
 \end{equation}
On the other hand, if $(u_{n_i})_i$ is such that $n_{i+1}-n_{i}>1$, then it follows that
\begin{equation}
\nrm{\sum_i a_i u_{n_i}}_J\ge (\sum_i a_i^2)^{1/2}.
\end{equation}  	
Since the unit basis of $\ell_2$ is spreading it follows that every such subsequence is 1-equivalent to the unit basis of $\ell_2$. Hence,
\begin{equation}
\mathbb E \nrm{\sum_i \vep_i a_i u_{2i}}_J=\mathbb E \nrm{\sum_i \vep_i a_i u_{2i+1}}_J=(\sum_{i} a_i^2)^{1/2}.
\end{equation}
Consequently,
\begin{equation} \label{dfsssfdsdfddd333ds}
\mathbb E \nrm{\sum_i a_i u_i}_J=
\mathbb E \nrm{\sum_i \vep_i a_{2i} u_{2i}}_J+ \mathbb E \nrm{\sum_i \vep_i a_{2i+1} u_{2i+1}}_J \approx (\sum_i a_i^2)^{1/2}.
\end{equation}
Now it follows from \eqref{dfsfdsdfddd333ds} and  \eqref{dfsssfdsdfddd333ds} that  the unit basis of $J$ is RUD with constant $\sqrt{2}$.
\end{proof}

This fact also shows that spaces with RUD bases  need not be embeddable into a space with unconditional basis.  Note that there is an analogous situation if we replace the role of the $\ell_2$ in the construction of James
space by an unconditional basis.
\eje
Let $(x_n)$ be an unconditional basis for the space $X$, and let $J_X$ be the generalized James space, which
is the completion of $c_{00}$ under the norm
$$
\|(a_n)\|_{J_X}=\sup\Big\{\Big\|\sum_{k=1}^{m-1}(a_{p_k}-a_{p_{k+1}})x_k\Big\|_X:p_1<\cdots<p_m\Big\}.
$$
The unit vector basis $(u_n)$ of $J_X$ is RUD. If $(x_n)$ is not equivalent to the $c_0$-basis, then $(u_n)$
is not unconditional. If in addition the basis $(x_n)_n$ is spreading, then
\begin{equation}
\mathbb E\nrm{\sum_n a_n u_n}_{J_X}\approx \nrm{\sum_n a_n x_n}_X.
\end{equation}
\feje

\begin{proof}
Fix a sequence of scalars $(a_i)_{i=1}^m$ and let $p_1<p_2<\cdots<p_n$ be such that
$$
\|\sum_{i=1}^ma_iu_i\|_{J_X}=\nrm{\sum_{j=1}^n(a_{p_j}-a_{p_{j+1}})x_j}_{X}.
$$
Now let $\tau:\{-1,+1\}^m\rightarrow \{-1,+1\}^m$ be defined in the following way: For
$\Theta=(\theta_i)_{i=1}^m$, let $\tau(\Theta)=(\theta'_i)_{i=1}^m$ be given by
$$
\theta'_i=
\left\{
\begin{array}{cl}
 \theta_i &\textrm{ if }i\notin\{p_1,\ldots,p_{n+1}\}   \\
 (-1)^j\theta_{p_j} & \textrm{ if } i=p_j.
\end{array}
\right.
$$
Now using that
$$
|a_{p_j}-a_{p_{j+1}}|\leq \max\{|\theta_{p_j}a_{p_j}-\theta_{p_{j+1}}a_{p_{j+1}}|, |\theta'_{p_j}a_{p_j}-\theta'_{p_{j+1}}a_{p_{j+1}}|\}
$$
and the fact that $(x_n)_n$ is  $C$-unconditional, we have
\begin{align*}
\|\sum_{i=1}^m a_iu_i\|_{J_X} = & \nrm{\sum_{j=1}^n(a_{p_j}-a_{p_{j+1}})x_j}_X \leq  C\nrm{\sum_{j=1}^n(\theta_{p_j}a_{p_j}-\theta_{p_{j+1}}a_{p_{j+1}})x_j}_X+  \\
+& C \nrm{\sum_{j=1}^n(\theta'_{p_j}a_{p_j}-\theta'_{p_{j+1}}a_{p_{j+1}})x_j}_X   \le\\
\le & C( \nrm{\sum_{i=1}^ma_i\theta_iu_i }_{J_X}+\nrm{\sum_{i=1}^ma_i\theta'_iu_i}_{J_X}).
\end{align*}
Since this holds for every choice of $(\theta_i)_{i=1}^m$ and $\tau$ is an involution
($\tau(\tau(\Theta))=\Theta$), taking averages at both sides gives us
$$
\|\sum_{i=1}^m a_iu_i\|_J \leq 2C\, \mathbb{E} \Big(\Big\|\sum_{i=1}^m \theta_i a_i u_i \Big\|_J  \Big).
$$
Now, to
check that $(u_n)$ is not unconditional, note that for every $k\in\mathbb N$, we have $\|\sum_{n=1}^k
u_n\|_{J_X}=1$, while
$$
\Big\|\sum_{n=1}^{2k} (-1)^n u_n\Big\|_{J_X}\geq \Big\|\sum_{n=1}^k x_n\Big\|_X.
$$
Hence, if $(u_n)$ were unconditional, then there would be a constant $C>0$ such that $\|\sum_{n=1}^k
x_n\|_X\leq C$. This is imposible because $(x_n)$ is not equivalent to the unit basis of $c_0$.
\end{proof}

\eje	
The well-known twisted sum of $\ell_2$ with $\ell_2$ by N. Kalton and N. Peck \cite{KP} has a natural 2-dimensional unconditional f.d.d. but it does not have an unconditional basis. Hence, by Corollary \ref{ufdd} it has a conditional RUD basis. In general, the non-trivial twisted sum of two spaces with unconditional bases gives also examples of conditional RUD bases.
\feje

We will see later (Theorem \ref{ri RUD}) that every block sequence of the Haar system on a rearrangement
invariant space with finite  upper Boyd index is RUD. In particular, the Haar basis in $L_1(0,1)$ is another
example of a RUD basis which is not unconditional. We also have the following:

\eje
The Walsh basis in $L_1[0,1]$ is RUD.
\feje

\begin{proof}
Recall that the Walsh basis is the canonical extension of the sequence of Rademacher functions $(r_n)$ to an
orthonormal basis of $L_2[0,1]$. Namely, for every finite set $s\subset\mathbb{N}$ we denote
$$
w_s=\Pi_{j\in s}r_j.
$$

Since, $(w_s)_s$ are orthonormal in $L_2[0,1]$, it follows that
$$
\Big\|\sum_s a_sw_s\Big\|_1\leq\Big\|\sum_s a_sw_s\Big\|_2=\Big(\sum_s a_s^2\Big)^\frac12.
$$
Now, since $(w_s)_s$ are also normalized in $L_1[0,1]$, and this space has cotype 2, it follows that
$$
\mathbb{E}\Big\|\sum_s a_s\vep_s w_s\Big\|_1\gtrsim \Big(\sum_s a_s^2\Big)^\frac12\geq \Big\|\sum_s a_sw_s\Big\|_1.
$$
Hence, $(w_s)_s$ is RUD.
\end{proof}

\eje
The Rademacher functions in $BMO[0,1]$ are a RUC basic sequence.
\feje

\begin{proof}
Recall the norm of the space $BMO[0,1]$ is given by
$$
\|f\|_{BMO[0,1]}=\sup_{I\subset[0,1]}\frac{1}{\lambda(I)}\int_I \Big|f-\frac{1}{\lambda(I)}\int_I fd\lambda\Big|d\lambda,
$$
where $\lambda$ denotes Lebesgue's measure on $[0,1]$. It is easy to check that for the Rademacher functions
$(r_n)$ we have
$$
\Big\|\sum_n a_nr_n\Big\|_{BMO[0,1]}=\Big(\sum_n a_n^2\Big)^\frac12+\sup_n\Big|\sum_{k=1}^n a_k\Big|.
$$
Hence, using the computations given in the proof of Example \ref{summing}, we have
$$
\mathbb{E}\Big\|\sum_n a_n\vep_n r_n\Big\|_{BMO[0,1]}\leq 3\Big(\sum_n a_n^2\Big)^\frac12\leq 3\Big\|\sum_n a_n r_n\Big\|_{BMO[0,1]}.
$$
\end{proof}

\eje
A conditional RUC basis  of $\ell_p$ and a conditional RUD basis of $\ell_p$ for $1<p<\infty$.
\feje
\begin{proof}
Let $(x_n)_n$ and $(y_n)_n$ be a Besselian non-Hilbertian, and Hilbertian non-Besselian bases of $\ell_2$,
respectively. Find a sequence of successive intervals $(I_k)_k$ such that $\bigcup_k I_k=\mathbb N$ and that
$(x_i)_{i\in I_k}$ and $(y_i)_{i\in I_k}$ are not $k$-Hilbertian and not $k$-Besselian, respectively. Since
$\langle x_i \rangle_{i\in I_k}$ and $\langle y_i \rangle_{i\in I_k}$ are (isometrically) finite dimensional
Hilbert spaces, of dimensions $d_k$ and $l_k$ respectively, and since $(\bigoplus_k\ell_2^{d_k})_{\ell_p},
(\bigoplus_k\ell_2^{l_k})_{\ell_p}$ are isomorphic to $\ell_p$, for $1<p<\infty$, the sequences $(x_i)_{i }$
and  $(y_i)_{i}$ are, in the natural ordering, bases of $\ell_p$. On the other hand, given scalars $(a_i)_i$,
one has that
\begin{align*}
\mathbb E_\vep \nrm{\sum_{k}\sum_{j\in I_k}\vep_j a_j x_j}\approx &(\mathbb E_\vep \nrm{\sum_{k}\sum_{j\in I_k}\vep_j a_j x_j}^p)^{\frac1p}=
(\mathbb E_\vep  \sum_{k}(\nrm{\sum_{j\in I_k}\vep_j a_j x_j}_{2})^p)^{\frac1p}=\\
=&( \sum_{k}\mathbb E_\vep (\nrm{\sum_{j\in I_k}\vep_j a_j x_j}_{2})^p)^{\frac1p}\approx
(\sum_{k} (\mathbb E_\vep \nrm{\sum_{j\in I_k}\vep_j a_j x_j}_{2})^p)^{\frac1p}=\\
=&(\sum_k (\sum_{j\in I_k}a_j^2)^{\frac p2})^{\frac1p}
\end{align*}
and similarly
\begin{align*}
\mathbb E_\vep \nrm{\sum_{k}\sum_{j\in I_k}\vep_j a_j y_j}\approx &(\sum_k (\sum_{j\in I_k}a_j^2)^{\frac p2})^{\frac1p}
\end{align*}
Hence, since $(x_n)_n$ is a Besselian  basis of $\ell_2$ , it follows that
\begin{align*}
\mathbb E_\vep \nrm{\sum_{k}\sum_{j\in I_k}\vep_j a_j x_j}\approx &(\sum_k (\sum_{j\in I_k}a_j^2)^{\frac p2})^{\frac1p} \lesssim  \nrm{\sum_{k}\sum_{j\in I_k}  a_j x_j}
\end{align*}
So, $(x_i)_i$ is a conditional RUC basis of $\ell_p$. And since $(y_n)_n$ is a Hilbertian basis of $\ell_2$,
it follows that
\begin{align*}
\mathbb E_\vep \nrm{\sum_{k}\sum_{j\in I_k}\vep_j a_j x_j}\approx &(\sum_k (\sum_{j\in I_k}a_j^2)^{\frac p2})^{\frac1p} \gtrsim  \nrm{\sum_{k}\sum_{j\in I_k}  a_j x_j}
\end{align*}
So, $(y_i)_i$ is a conditional RUD basis of $\ell_p$.
\end{proof}

There are further examples that have been considered in the literature. For instance, in \cite{KS} it is shown that the Olevskii system, an orthonormal system which is simultaneously a basis in $L_1[0,1]$ and a basic sequence in $L_\infty[0,1]$, forms an RUC basis in $L_p[0,1]$ if and only if $2\leq p<\infty$. In fact, this is an RUC basis of every rearrangement invariant (r.i.) space $X$ with finite cotype and upper Boyd index $\beta_X<1/2$ \cite[Theorem 1]{KS}. These results are extended in \cite{DSS} where the authors study conditions for an r.i. space to have a complete orthonormal uniformly bounded RUC system.

In the non-commutative setting there are also interesting examples of RUC bases. For instance, in the space $C^p$ (compact operators $a:\ell_2\rightarrow\ell_2$ such that $\sigma_p(a)=(tr(aa^*)^{p/2})^{1/p}<\infty$) it is well-known that the canonical basis $(e_n\otimes e_m)_{n,m=1}^\infty$ is not unconditional for $p\neq2$. However, for $2\leq p<\infty$, $(e_n\otimes e_m)_{n,m=1}^\infty$ is a RUC basis \cite[Theorem 3.1]{BKPS}. Hence, by Proposition \ref{duality} and the duality between $C^p$ and $C^{p/p-1}$, it follows that for $1<p\leq2$,  $(e_n\otimes e_m)_{n,m=1}^\infty$ is a RUD basis (which of course cannot be RUC). Surprisingly enough, in \cite{Garling-Tomczak} it was shown that the space $C^p$ also has a RUC basis for $1\leq p\leq2$.

More examples in the non-commutative context can be found in \cite{DS}. Also, in \cite{Witvliet}, the connection between R-boundedness, UMD spaces and RUC Schauder decompositions is explored.

\section{Uniqueness of bases}

Another point worth dwelling on is the uniqueness of RUD or RUC basis on some Banach spaces. Concerning
unconditionallity, it is well known that the only Banach spaces with a unique unconditional basis (up to
equivalence) are $\ell_1$, $\ell_2$ and $c_0$ (cf. \cite{LT1}). Using \cite[Prop. 2.1]{BKPS}, one can see
that every RUC basis in $\ell_1$ must be equivalent to the unit vector basis (See Theorem \ref{iueiuhuitrtr}
below).
Note also that there are RUC basis of $c_0$ which are not RUD (see \cite[Prop. 2.2]{BKPS}, or use the
construction of \cite{Wojtaszczyk} starting with the summing basis of $c_0$).

In $\ell_2$ we can find bases which are RUD but not RUC, or viceversa. Indeed, for every basis $(e_n)$ in
$\ell_2$, using the parallelogram law we know that
$$
\mathbb{E} \Big\|\sum_{i=1}^m \epsilon_i a_i e_i \Big\|^2 =\sum_{i=1}^m a_i^2.
$$
\begin{definition}
A basis $(x_n)_n$ is called Besselian  if there is a constant $K>0$ such that
\begin{equation}
(\sum_n a_n^2)^{\frac12} \le K\nrm{\sum_n a_n x_n} \text{ for every sequence of scalars $(a_n)_n$.}
\end{equation}
A basis $(x_n)_n$ is called Hilbertian  if there is a constant $K>0$ such that
\begin{equation}
\nrm{\sum_n a_n x_n}\le K (\sum_n a_n^2)^{\frac12} \text{ for every sequence of scalars $(a_n)_n$.}
\end{equation}
\end{definition} Thus, every non-Besselian (respectively non-Hilbertian) basis of $\ell_2$ is not RUC (resp.
RUD). A combination of a non RUD basis with a non RUC one yields a basis of $\ell_2$ which fails both
properties.

\teor[P. Billard, S. Kwapie\'n, A. Pelczy\'nski and Ch. Samuel \cite{BKPS}] \label{iueiuhuitrtr}
Every  RUC basis of $\ell_1$ is equivalent to the unit basis of $\ell_1$.
\fteor
\begin{proof}
Fix a RUC  basis $(x_n)_n$ of $\ell_1$ with constant $C$. Let $(x_n^*)_n$ be the biorthogonal sequence to
$(x_n)_n$.  Let $K$ be the cotype constant of $\ell_1$. Define the  operator $T:L_1([0,1],\ell_1)\to \ell_2$
defined by
$$ T(f):=\sum_{n=1}^\infty (\int_0^1 x_n^*(f(t))r_n(t) dt)u_n $$
for every $f\in L_1([0,1],\ell_1)$.  It is well-defined and bounded:
\begin{align*}
\nrm{T(f)}_2 = &\left( \sum_n \left(\int_0^1 x_n^*(f(t)) r_n(t)dt \right)^2 \right)^\frac12\le \int_0^1 \left( \sum_n \left( x_n^*(f(t)) r_n(t)\right)^2dt  \right)^\frac12 = \\
=&  \int_0^1 \left( \sum_n \left( x_n^*(f(t)) \right)^2dt  \right)^\frac12  \le  K \int_0^1  \mathbb E_\vep \nrm{\sum_n \vep_n x_n^*(f(t))  x_n}dt   \le \\
\le &  C\cdot K \int_0^1  \nrm{\sum_n x_n^*(f(t)) x_n}dt  = C\cdot K \int_0^1 \nrm{f(t)}dt =C\cdot K \nrm{f}
\end{align*}
Since $L_1([0,1],\ell_1)$ is a $\mathcal L_1$-space, it follows that the operator $T$ is absolutely summing,
with absolutely summing constant $K_G\nrm{T}$. It follows that for every sequence of scalars $(a_i)_{i=1}^n$
one has that
\begin{align*}
\sum_{i=1}^n |a_i|=&\sum_{i=1}^n \nrm{T(a_i r_i(\cdot) x_i)}\le K_G \nrm{T} \max_{\vep} \nrm{\sum_{i=1}^n  \vep_i a_i r_i(\cdot) x_i}\le\\
\le &  K_G \cdot C \cdot K \nrm{\sum_{i=1}^n  a_i r_i(\cdot) x_i}
\le  K_G \cdot C^2 \cdot K \nrm{\sum_{i=1}^n  a_i x_i}.
\end{align*}
\end{proof}

\begin{corollary}
A Banach space  has a unique (up to equivalence) RUC basis if an only if it is isomorphic to $\ell_1$.
\end{corollary}
\begin{proof}
The previous Theorem \ref{iueiuhuitrtr} proves that $\ell_1$ has a unique RUC basis. Suppose now that $X$ is
a space with the same property. Fix a RUC basis $(x_n)_n$ of $X$. It follows that $(\vep_n x_n)_n$ is a RUC
sequence of $X$ for every sequence $(\vep_n)_n$ of signs. Hence, by hypothesis, it is equivalent to $(x_n)_n$
a simple uniform boundedness principle shows that there is a constant $K$ such that
$$\nrm{\sum_n a_n x_n}\le K \nrm{\sum_n \vep_n a_n x_n}$$
for every sequence of scalars. Hence, $(x_n)_n$ is the unique unconditional basis of $X$. It follows then
that $X$ is isomorphic to either $c_0$, $\ell_1$ or $\ell_2$. We have already said that $c_0$ and $\ell_2$
have conditional RUC bases.
\end{proof}

Theorem \ref{iueiuhuitrtr} also motivates the following question: Is every basis of $\ell_1$ a RUD basis? It is not hard to check that every triangular basis of $\ell_1$ is RUD (in particular, every Bourgain-Delbaen basis of $\ell_1$ is RUD). The same question for $L_1$ is also open.

\subsection{Uniqueness of RUD bases}

\teor
Every Banach space with an RUD basis has two non-equivalent RUD bases.
\fteor
The proof has two parts.
\lema
Suppose that $X$ is a space with an RUD basis and not isomorphic to $c_0$. Then $X$ has two non-equivalent
RUD bases.
\flema
\begin{proof}As in the proof of the previous corollary such
space has a unique unconditional basis; hence it must be isomorphic to $c_0$, $\ell_1$ or $\ell_2$.  It
cannot be $c_0$ by hypothesis, or $\ell_2$ as this space has a Hilbertian conditional basis; in $\ell_1$ the
sequence $(x_n)_n$ defined by $x_0=u_0$, $x_{n+1}=u_{n+1}-u_n$ is a  conditional basis of $\ell_1$ such that
$\mathbb E_\vep\nrm{\sum_n a_n x_n}\approx \sum_n |a_n|$, hence RUD.
\end{proof}
The next is the key result
\lema\label{iuuirtgbjbjkgf322}
$c_0$ has two non-equivalent RUD bases.
\flema

The proof of this Lemma is based on the Bourgain-Delbaen construction of  $\mathcal L_\infty$ spaces with the Schur property in  \cite{Bourgain-Lp, BD}, and we follow the exposition and notation of \cite{Haydon}.  In fact, the authors construct for arbitrarily large $n$ a basis $(d_i)_{i=1}^n:=(d_i^{(n)})_{i=1}^n$ of $\ell_\infty^n$  a partition $A\cup B\cup C=\{1,\dots,n\}$ and a constant $K$ independent on $n$ such that
\begin{enumerate}
\item[(i)] $(d_i)_{i\in A}$, $(d_i)_{i\in B}$ and $(d_i)_{i\in C}$ are $K$-equivalent to the unit basis of  $(\oplus_{j=1}^k \ell_\infty^{n_j})_{\ell_1}$, $(\oplus_{j=1}^l \ell_\infty^{m_j})_{\ell_1}$ and of $\ell_\infty(r)$ respectively.
\item[(ii)] $k$ and $l$ grow to infinity as $n$ grows to infinite.
\end{enumerate}
It follows then from (i) that the basis $(d_i)_{i=1}^n$ is at most $K$-equivalent to the unit basis of $\ell_\infty^n$. Now the canonical basis $(d_i)_i$ of $c_0=(\oplus_n \ell_\infty^n)_{c_0}$ extending each  $(d_i^{(n)})_{i=1}^n$ cannot be, by the condition (ii), equivalent to the unit basis of $c_0$. On the other hand,  it will follow from Proposition \ref{hohriogihohgfhg} that $(e_i)_i$ is RUD.

 We will begin by recalling the badly unconditional RUD-bases of $\ell_\infty^n$. Fix $\la>1$, and $b<1/2$ such that
$$ 1+2b\la\le \la.$$
Let $\De_0:=\{0\}$; Suppose defined $\De_{n}$, and set  $\Ga_n:=\bigcup_{k\le n}\De_n$. Let $\De_{n+1}$ be
the collection of all quintuples $(m,\vep_0,\vep_1,\sig_0,\sig_1)$ such that $\vep_0,\vep_1\in \{-1,1\}$,
$\sig_0\in \Ga_m$ and $\sig_1\notin \Ga_m$. Let $\Ga_{n+1}:=\Ga_n\cup \De_{n+1}$. For every $n$, fix a total
ordering $\prec_n$ of the finite set  $\De_{n}$, and let $\prec^n$ be the total ordering on $\Ga_n$ extending
the fix orderings $\prec_n$ on each $\De_m$, and such that each element of $\De_{m}$ is strictly smaller than
each element of $\De_{m+1}$.

We define vectors $(d_{\sig}^*)_{\sig\in \Ga_{n}}\subset \ell_1(\Ga_{n+1})$ and $(d_{\sig}^{(n)})_{\sig\in
\Ga_{n}}\subset \ell_\infty(\Ga_n)$ with the following properties:
\begin{enumerate}
\item[(a)] $(d_{\sig}^*,d_{\sig}^{(n)})_{\sig\in \Ga_n}$ is a biorthogonal
 sequence with $1\le \nrm{d_{\sig}^*}_{\ell_1}, \nrm{d_{\sig}^{(n)}}_\infty\le \la$.
\item[(b)]  $(d_\sig^*)_{\sig\in \Ga_n}$, ordered by $\prec_n$,  is a Schauder basis of $\ell_1(\Ga_n)$ with basis constant $\le \la$.

\end{enumerate}
The construction of $(d_\sig^*,d_{\sig}^{(n)})$, $\sig\in \Ga_n$, is done inductively on $n$: For $n=0$, let
$d_{0}^*=d_{0}^{(0)}:=u_0$. Suppose all done for $n$, and for each $m\le n$, let $P_m^*: \ell_1(\Ga_n)\to
\ell_1(\Ga_m)$ be the canonical projection
$$P_m^*:=\sum_{\tau \in \Ga_m} d_\tau^{(n)}\otimes d_\tau^*$$
of norm $\le\la$ associated to the basis $(d_\tau^*)_{\tau\in \Ga_n}$. For $\sig\in \De_{n+1}$,
$\sig=(m,\vep_0,\vep_1,\sig_0,\sig_1)$, let
\begin{align*}
d_\sig^*:= & u_\sig^*-c_{\sig}^* \\
c_\sig^*:= & \vep_0 u_{\sig_0}^*+ \vep_1 b(u_{\sig_1}^*-P_{m}^* u_{\sig_1}^*)\\
d_\sig^{(n+1)}:= & u_{\sig}.
\end{align*}
Let $\tau\in \Ga_{n}$. Then let
\begin{align*}
d_\tau^{(n+1)}:=& d_\tau^{(n)}+\sum_{\sig\in \De_{n+1}}\langle c_{\sig}^*, d_\tau^{(n)} \rangle u_\sig^*.
\end{align*}
Observe that for $\sig\in \De_m$,
\begin{equation}\label{nmfngjfngd}
d_\sig^{(n)}\upharpoonright \De_m=u_\sig.
\end{equation}
For each $m\le n$, let $D_m^{(n)}:=\langle d_{\sig}^{(n}\rangle_{\sig\in \De_m}$.

\prop\label{ioiuohoi4h543fdbf} For every $m\le n$ and every sequence of scalars $(a_\sig)_{\sig\in \De_m}$ one has that
\begin{equation}\label{jnjkruiuiee}
\max_{\sig\in \De_m} |a_\sig|\le\nrm{\sum_{\sig \in \De_m}a_\sig d_\sig^{(n)}\upharpoonright  \De_m }_\infty \le    \nrm{\sum_{\sig \in \De_m}a_\sig d_\sig^{(n)} }_\infty\le \la\max_{\sig\in \De_m} |a_\sig|.
\end{equation}
\fprop

\begin{proof}
Set $x:=\sum_{\sig \in \De_m}a_\sig d_\sig^{(n)}$. It follows from \eqref{nmfngjfngd} that $(x)_\sig= a_\sig$
for every $\sig\in \De_m$; hence, $\nrm{x\upharpoonright \De_m}_\infty\ge \max_{\sig\in \De_m}|a_\sig|$.

Let $\tau\in \De_k$ with $k\le n$.  If $k<m$, then $ (x)_\tau= 0$ because $(d_\sig^{(n)})_\tau=0$    because
$(d_\sig^{(n)})_\tau=0$ for every $\sig\in \De_m$.  If $k=m$, then $(x)_\tau=a_\tau$ because of
\eqref{nmfngjfngd}. Suppose that $m<k\le n$. We prove by induction on $k$ that
\begin{equation}
|(x)|_\tau\le \la\max_{\sig\in  \De_m}|a_\sig|.
\end{equation}
 $\tau=(l,\vep_0,\vep_1,\tau_0,\tau_1)$ with $l<k$. Then
$$(x)_\tau= \langle d_\tau^*+c_\tau^*,x\rangle=\langle c_\tau^*, x\rangle.$$
Suppose first that $l< m$. Then
\begin{align*}
|\langle c_\tau^*, x\rangle|\le  & | (x)_{\tau_0}|+ b |((I_n-P_{l,n})x)_{\tau_1}|= b |(x)_{\tau_1}|\le \la b \max_{\sig\in \De_m}|a_\sig| \le \la \max_{\sig \in \De_m} |a_\sig|.
\end{align*}
If $l\ge m$, then
\begin{align*}
|\langle c_\tau^*, x\rangle|= &   |(x)_{\tau_0}|+  0 \le \la\max_{\sig\in \De_m}|a_\sig|.
\end{align*}
\end{proof}

\prop \label{ioiodsdggfss}
Let $m_0<m_1<\cdots<m_{l}<n$. Then
$$\frac1\la\nrm{\sum_{i=0}^l\sum_{\sig\in \De_{m_i}}a_\sig d_{\sig}^{(n)}   }_\infty \le
\sum_{i=0}^l\max_{\sig\in \De_m}|a_\sig| \le \frac{1}{b} \nrm{\sum_{i=0}^l\sum_{\sig\in \De_{m_i}}a_\sig d_{\sig}^{(n)}   }_\infty,  $$
for every sequence of scalars $(a_\sig)_{\sig\in \bigcup_{i\le l}\De_{m_i}}$.
\fprop

\begin{proof}
The first inequality: Using \eqref{jnjkruiuiee} in Proposition  \ref{ioiuohoi4h543fdbf},
\begin{align*}
\nrm{\sum_{i=0}^l\sum_{\sig\in \De_{m_i}}a_\sig d_{\sig}^{(n)}   }_\infty \le \sum_{i=0}^l\nrm{\sum_{\sig\in \De_{m_i}}a_\sig d_{\sig}^{(n)} }_\infty \le \la  \sum_{i=0}^l\max_{\sig\in \De_{m_i}}|a_\sig|.
\end{align*}
For the second inequality: For each $i\le l$, let $\sig_i\in \De_{m_i}$ and $\vep_i\in \{-1,1\}$  be such
that $ \vep_i a_{\sig_i}=\max_{\sig\in \De_{m_i}}|a_\sig|$. We also suppose that $l\geq 1$, since otherwise
there is nothing to prove.  For each $0<i\le l$   We define recursively $\tau_i\in \De_{m_i+1}$  as follows.
Let $\tau_1:=  (m_0,\vep_0,\vep_1,\sig_0,\sig_1)\in \De_{m_1+1} $.  Let
$\tau_2:=(m_1,1,\vep_2,\tau_1,\sig_2)\De_{m_2+1}$; in general, let $\tau_{i}:=(m_{i-1}
,1,\vep_{m_i},\tau_{i-1},\sig_{i})$.  Set $x_{i}:=\sum_{\sig\in \De_{m_i}}a_\sig d_{\sig}^{(n)}$ for each
$i\le l$, and    $x:=\sum_{i=0}^lx_i$. Let us prove inductively that for every $0<i\le l$ one has that
$$ (x)_{\tau_{i}}= |a_{\sig_0}|+b\sum_{j=1}^{i}|a_{\sig_{j}}|= \max_{\sig\in \De_{m_0}}|a_{\sig}|+b \sum_{j=1}^{i}\max_{\sig\in \De_{m_j}}|a_{\sig}|: $$
Suppose that $i=1$. Then, using that $\tau_{1}\in \De_{m_1+1}$ implies that $\langle
d_{\tau_1}^*,x\rangle=0$, it follows that
\begin{align*}
(x)_{\tau_{1}}=& \langle u_{\tau_1}^*,x\rangle=\langle d_{\tau_1}^*+c_{\tau_1}^*,x\rangle =\langle c_{\tau_1}^*,x\rangle=  \vep_0 (x)_{\sig_0}+ b \vep_1 (x-P_{m_0}x)_{\sig_1}= \\
= & |a_{\sig_0}|+ b\vep_1( \sum_{j=0}^lx_j)_{\sig_1}=  |a_{\sig_0}|+ b\vep_1( x_1)_{\sig_1}=  |a_{\sig_0}|+ b|a_{\sig_1}|.
\end{align*}
Suppose that $\tau_{i}\in \De_{m_i+1}$ is such that
$(x)_{\tau_{i}}=|a_{\sig_0}|+b\sum_{j=1}^{i}|a_{\sig_{j}}|$. Then,
\begin{align*}
(x)_{\tau_{i+1}}=& \langle u_{\tau_{i+1}}^*,x\rangle=\langle d_{\tau_{i+1}}^*+c_{\tau_{i+1}}^*,x\rangle =\langle c_{\tau_{i+1}}^*,x\rangle=   (x)_{\tau_{i}}+ b \vep_{i+1} (x-P_{m_{i+1}}x)_{\sig_{i+1}}= \\
= & |a_{\sig_0}|+ b\sum_{j=1}^{i}|a_{\sig_{j}}|+\vep_{i+1}b( x_{i+1})_{\sig_{i+1}}=  |a_{\sig_0}|+ b\sum_{j=1}^{i+1}|a_{\sig_{j}}|.
\end{align*}

\end{proof}

\prop
The basis $(d_\sig^{(n)})_{\sig\in \Ga_n}$ of $\ell_\infty(\Ga_n)$ is RUD with constant $\le \la(2/b+1)$.
\fprop
\begin{proof}
Let $A:=\bigcup_{m<n,m\text{ even}}\De_m$, $B:=\bigcup_{m<n, m\text{ odd}}\De_m$ and $C:=\De_n$. Then, by
Proposition \ref{ioiodsdggfss},  $(d_\sig^{(n)})_{\sig \in A}$ and $(d_\sig^{(n)})_{\sig \in B}$ are $\la/b$
equivalent to the unit vector basis of $(\sum_{m\in A}\ell_\infty(\De_m))_{\ell_1}$ and of  $(\sum_{m\in
B}\ell_\infty(\De_m))_{\ell_1}$ respectively. Since these two unit vector bases are 1-unconditional, the
subsequences $(d_\sig^{(n)})_{\sig \in A}$ and $(d_\sig^{(n)})_{\sig \in B}$  are unconditional  with
constant $\le \la/b$.  Also, it follows from Proposition  \ref{ioiuohoi4h543fdbf} that  $(d_\sig^{(n)})_{\sig
\in C}$ is $\la$-equivalent to the unit vector basis of $\ell_\infty(C)$, hence unconditional with constant
$\le \la$. The desired result follows from Proposition \ref{hohriogihohgfhg} (1).

\end{proof}

We are ready to prove Lemma \ref{iuuirtgbjbjkgf322}.
\begin{proof}
For each $n$ let $\Ga_n$ be the finite sets defined above, and let $\Ga:=\bigcup_n \Ga_n$, disjoint union.
Then  $(\sum_{n\in \N} \ell_\infty(\Ga_n))_\infty$ is  isometric to $c_0(\Ga)$, which in turn is isometric to
$c_0$. We order  $\Ga$ canonically by first consider the total ordering $\prec_n$ as above and then declaring
that each element of $\Ga_m$ strictly smaller than each element of $\Ga_n$    for $m<n$. Then
$(d_\sig^{(n)})_{n\in \N,\sig\in \Ga_n}$ is a basis of  $(\sum_{n\in \N} \ell_\infty(\Ga_n))_\infty$ which is
RUD with constant $\le\la(2/b+1)$. On the other hand, this basis has arbitrary long subsequences
$\la/b$-equivalent to the unit vector basis of $\ell_1$, hence it cannot be equivalent to the unit vector
basis of $c_0$.
\end{proof}

Note this construction also provides an example of a basis $(x_n)$ such that both $(x_n)$ and its biortogonal functionsl $(x_n^*)$ are RUD, but $(x_n)$ is not unconditional.

\section{RUC, RUD and unconditional bases}\label{nounc}
It is not true that every basic sequence has a RUC or a RUD subsequence as the summing basis of $c_0$ shows. However, it is well-known that weakly-null sequences have always subsequences with some sort of partial unconditionality such as Elton's or Odell's unconditionality (see \cite{Elton}, \cite{Odell}). It is natural then to ask if weakly-null sequences have subsequences with partial random unconditionality RUC or RUD. We are going to prove that the Maurey-Rosenthal example of a weakly-null basis without unconditional subsequences has the stronger property of not having RUD subsequences.

Secondly,  we will see that RUC or RUD basic sequences do not necessarily have unconditional subsequences. Interestingly, the Johnson, Maurey, and Schechtman example of a weakly-null sequence in $L_1[0,1]$ without unconditional subsequences have a RUD subsequence as this is the case not only for $L_1[0,1]$ but also for many rearrangement invariant spaces on $[0,1]$ (see Theorem
\ref{ri RUD}). Observe that this subsequence gives an example of a weakly-null sequence without RUC subsequences. And a simple modification of the Maurey-Rosenthal example gives a RUC sequence without unconditional subsequences.

Finally, we will give an example of a RUD sequence that has a non-RUD block-subsequence; the analogue for RUC sequences can be found by taking a RUC basis of $C[0,1]$, that always exist by a result of  Wojtaszczyk in \cite{Wojtaszczyk}.

Let us first introduce some useful notation, which we will use to introduce not only the Maurey-Rosenthal example but also the ulterior examples. Given any finite set $s\subset\mathbb{N}$ of even cardinality,
let
$$
\mathcal{E}(s)=\{(\varepsilon_i)_{i\in s}\in\{-1,1\}^s:\sharp\{i\in s:\varepsilon_i=1\}=\sharp\{i\in s:\varepsilon_i=-1\}\}.
$$
This set consists of all equi-distributed signs indexed on a given set $s$. Let
$k_m=\sharp\mathcal{E}(\{1,\ldots,m\})$. Notice that the cardinality of a set $\mathcal{E}(s)$ only depends
on the cardinality of $s$, so $\sharp\mathcal{E}(s)=k_m$ for any set $s$ with $\sharp s=m$. From the central
limit theorem it follows that
$$
\lim_{m\rightarrow\infty}\frac{k_m}{2^m}=1.
$$

Maurey-Rosenthal's space $Z_{MR}$ can be described as follows: Given $\delta\in(0,1)$, take an increasing
sequence $M=\{m_n\}$ so that
\begin{equation} \label{oiojihsdfoijsdswebl}
\sum_j\sum_{k\neq j}\sqrt{\min\{\frac{m_j}{m_k},\frac{m_k}{m_j}\}}\leq\delta,
\end{equation}
and fix a one-to-one function
$$
\sigma:\mathbb{N}^{<\infty}\rightarrow \{m_k\}_{k\ge 2}
$$
such that $\sigma(s)>\sharp s$. Let
$$
\mathcal{B}_0=\{(s_1,\ldots,s_n):s_1\in\mathcal{S},\,s_1<\ldots<s_n,\,\sharp s_j\in M,\, \sharp s_{i+1}=\sigma(s_1\cup\ldots\cup s_i)>\sharp s_i\}.
$$
Let $u_n$ denote the n-th unit vector in $c_{00}$ and $u_n^*$ its bi-orthogonal functional. Let us consider
the set
$$
\mathcal{N}_0=\{\sum_{i=1}^n\frac{1}{\sqrt{\sharp s_i}}\sum_{j\in s_i}u_j^*: (s_1,\ldots, s_n)\in\mathcal{B}_0\}
$$
and define $Z_{MR}$ as the space of scalar sequences $(a_n)_{n=0}^\infty$ such that
$$
\|(a_n)\|_{Z_{MR}}=\sup\{|\langle \phi,\sum_n a_n u_n\rangle|:\phi\in\mathcal{N}_0\}\vee\sup_{n\in\mathbb{N}}|a_n|<\infty.
$$

\lema
Let $(s_1,\dots,s_n)\in \mathcal B_0$. Then
\begin{equation} \label{jkmncsaeerer}
\mathbb E_\vep \nrm{\sum_{i=1}^n\frac1{(\sharp s_i)^\frac12}\sum_{k\in s_i}\vep_k u_k}\le 3.
\end{equation}
\flema
\begin{proof}
Fix $(s_1,\dots,s_n)\in \mathcal B_0$,   fix $(t_1,\dots,t_m)\in \mathcal B_0$, and set
$\varphi:=\sum_{j=1}^n (\sharp t_j)^{-1/2}\sum_{k\in t_j}u_k$, $x_\varepsilon:=\sum_{i=1}^n (\sharp
s_i)^{-1/2}\sum_{k\in s_i}\varepsilon_ku_k$, where $\varepsilon=(\varepsilon_k)_{k\in \bigcup_{i}s_i}$ is a
sequence of signs. Let
$$i_0=\min\{i\in\{1,\ldots,n\}:s_i\neq t_i\}. $$
Then
\begin{align*}
|\langle \vphi,x_\varepsilon \rangle|=& \sum_{j=1}^m\sum_{i=1}^n\frac1{(\sharp s_i \sharp t_j)^\frac12}\sum_{k\in s_i\cap t_j} \vep_k=
\sum_{j=1}^m\left(\sum_{\sharp s_i =\sharp t_j}\frac1{\sharp s_i}\sum_{k\in s_i\cap t_j} \vep_k+ \sum_{\sharp s_i \neq\sharp t_j}
\frac1{\sharp s_i}\sum_{k\in s_i\cap t_j} \vep_k \right)=\\
=&\sum_{j<i_0} \frac1{\sharp s_j}\sum_{k\in s_j}\vep_k +\frac{1}{\sharp s_{i_0}}\sum_{k\in s_{i_0}\cap t_{i_0}}\vep_k +\sum_{j\ge i_0}\sum_{i>i_0}
\frac1{(\sharp s_i \sharp t_j)^{\frac12}}\sum_{k\in s_i\cap t_j}\vep_k.
\end{align*}
It follows from \eqref{oiojihsdfoijsdswebl} that
\begin{equation}
|\sum_{j\ge i_0}\sum_{i>i_0}
\frac1{(\sharp s_i \sharp t_j)^{\frac12}}\sum_{k\in s_i\cap t_j}\vep_k|\le \sum_{j\ge i_0}\sum_{i>i_0}
\frac{\sharp(s_i\cap t_j)}{(\sharp s_i \sharp t_j)^{\frac12}}\le \de.
\end{equation}
Hence,
\begin{align*}
\nrm{x_\varepsilon} \le \max_{m=1}^n    |\sum_{i=1}^m \frac{1}{\sharp s_i}\sum_{k\in s_i}\varepsilon_k |+1+\de.
\end{align*}
Using this inequality and  Levy's inequality,  we obtain that
\begin{align*}
\mathbb E_\varepsilon \nrm{x_\varepsilon} \le & \mathbb E_\varepsilon(\max_{m=1}^n
|\sum_{i=1}^n \frac{1}{\sharp s_i}\sum_{k\in s_i}\varepsilon_k |)+1+\de \le 2\mathbb E_\varepsilon(
|\sum_{i=1}^m \frac{1}{\sharp s_i}\sum_{k\in s_i}\varepsilon_k |)+1+\de = \\
= &  2 \int_0^1 |\sum_{i=1}^n  \frac{1}{\sharp s_i}\sum_{k\in s_i}r_k(t) |dt +1+\de \le
2 \nrm{\sum_{i=1}^n\frac{1}{\sharp s_i}\sum_{k\in s_i}\varepsilon_k u_k }_2  +1+\de  = \\
= & (\sum_{i=1}^n
\frac1{\sharp s_i})^\frac12\le 1+2\de \le 3.
\end{align*}
\end{proof}

\teor\label{Maurey-Rosenthal}
The unit vector basis $(u_n)$ in the space $Z_{MR}$ is a weakly null sequence with no RUD subsequences.
\fteor

\begin{proof}
Let $N\subset\mathbb{N}$ be any infinite set. Given any $K>0$ we will see that $(u_n)_{n\in N}$ is not
$K$-RUD.  Let $n>3K$ and  $s_1,\ldots,s_n\subset N$ such that $(s_1,\ldots,s_n)\in\mathcal{B}_0$. Then
\begin{equation}
\nrm{\sum_{i=1}^n \frac1{(\sharp s_i)^\frac12}\sum_{k\in s_i}u_k}\ge \langle \sum_{i=1}^n \frac1{(\sharp s_i)^\frac12}\sum_{k\in s_i}u_k,\sum_{i=1}^n \frac1{(\sharp s_i)^\frac12}\sum_{k\in s_i}u_k
\rangle=n,
\end{equation}
while from \eqref{jkmncsaeerer} it we have that
\begin{equation}
\mathbb E_{\varepsilon}\nrm{\sum_{i=1}^n\frac1{(\sharp s_i)^\frac12}\sum_{k\in s_i}\vep_k u_k} \le 3.
\end{equation}
Hence $(u_n)_{n\in N}$ is not $K$-RUD.
\end{proof}

We present now a RUC sequence without unconditional subsequences. Given $(a_i)_{i=1}^n\in c_{00}$, let
$$\nrm{(a_i)_i}_{\mathrm{RUC}}:= \nrm{(a_i)_i}_{\mathrm{MR}}+ \mathbb E_{\varepsilon}\nrm{(\varepsilon_i a_i)_i}_{\mathrm{MR}}.$$
Let $Z_{\mathrm{RUC}}$ be the completion of $c_{00}$ under this norm.

\teor\label{MR-RUC}
The unit basis $(u_n)_n$ of $Z_{\mathrm{RUC}}$ is a weakly-null RUC basis without unconditional subsequences.
\fteor
\begin{proof}
It is RUC: We have that
$$\nrm{(a_i)_i}_{\mathrm{RUC}}\le \nrm{(a_i)_i}_{\mathrm{MR}}+\mathbb{E}_\varepsilon \nrm{(\vep_i a_i)_i}_{\mathrm{MR}}\le 2 \nrm{(a_i)_i}_{\mathrm{RUC}}.$$
Hence,
$$\mathbb E_\varepsilon\nrm{(\varepsilon_i a_i)_i}_{\mathrm{RUC}}\le
2\mathbb E_\varepsilon\nrm{(\varepsilon_i a_i)_i}_{\mathrm{MR}} \le  2 \mathbb
E_\varepsilon\nrm{(\varepsilon_i a_i)_i}_{\mathrm{RUC}}.$$ It follows that
$$\mathbb E_\varepsilon\nrm{(\varepsilon_i a_i)_i}_{\mathrm{RUC}}\le
2\mathbb E_\varepsilon\nrm{(\varepsilon_i a_i)_i}_{\mathrm{MR}} \le 2\nrm{(a_i)_i}_{\mathrm{RUC}} $$

Hence $(u_i)_i$ is 2-RUC. On the other hand, given $(s_i)_{i=1}^n\mathcal B_0$, we have from
\eqref{jkmncsaeerer} that
\begin{align*}
\nrm{\sum_{i=1}^n \frac{1}{(\sharp s_i)^\frac12}(-1)^i \sum_{k\in s_i}u_k}_{\mathrm{RUC}}=&
\nrm{\sum_{i=1}^n \frac{1}{(\sharp s_i)^\frac12}(-1)^i \sum_{k\in s_i}u_k}_{\mathrm{MR}} + \\
+ &
\mathbb E_{\vep}\nrm{\sum_{i=1}^n \frac{1}{(\sharp s_i)^\frac12}(-1)^i \sum_{k\in s_i}\vep_ku_k}_{\mathrm{MR}} \le
6.
\end{align*}
On the other hand $\nrm{\sum_{i=1}^n (\sharp s_i)^{-1/2}\sum_{k\in s_i}u_k}_{\mathrm{RUC}}\ge n$. Thus, it
has no unconditional subsequence.

\end{proof}

\subsection{RUD basis without unconditional subsequences}
We present now a weakly-null  RUD basis   without unconditional subsequences. Given a finite set $s$, let
$\mathcal{E}(s)$ be the collection of equi-distributed signs in $s$.

Let us fix $\delta\in (0,1)$. We will take an increasing sequence of even numbers $M=\{m_n\}$ so that
\begin{enumerate}
\item[(i)] $\sum_j\sum_{k\neq j}\sqrt{\min\{\frac{m_j}{m_k},\frac{m_k}{m_j}\}}\leq\delta$, and
\item[(ii)] $\prod_{n=1}^\infty\frac{k_{m_n}}{2^{m_n}}\geq 1-\delta$.
\end{enumerate}

Fix a one-to-one function
$$
\sigma:\mathbb{N}^{<\infty}\rightarrow M
$$
such that $\sigma(s)>\sharp s$. Let
$$
\mathcal{B}=\{(s_1,\ldots,s_n):s_1\in\mathcal{S},\,s_1<\ldots<s_n,\,\sharp s_j\in M,\, \sharp s_{i+1}=\sigma(s_1\cup\ldots\cup s_i)>\sharp s_i\}.
$$

Let $u_n$ denote the $n^{\mathrm{th}}$ unit vector in $c_{00}$ and $u_n^*$ its bi-orthogonal functional. Let
us consider the set
$$
\begin{array}{ll}
\mathcal{N}=&\{\sum_{i=1}^n\frac{1}{\sqrt{\sharp s_i}}\sum_{j\in s_i}\varepsilon_i(j)u_j^*: (s_1,\ldots, s_n)\in\mathcal{B}, \textrm{ with }\varepsilon_i\in\mathcal{E}(s_i), \,\forall i=1,\ldots,n\}\\
&\cup\{\pm\sum_{i=1}^n\frac{1}{\sqrt{\sharp s_i}}\sum_{j\in s_i}u_j^*: (s_1,\ldots, s_n)\in\mathcal{B}\}\cup\{\pm u_k^*:k\in \mathbb{N}\}.
\end{array}
$$

Now, we define $Z_{\mathrm{RUD}}$ as the space of scalar sequences $(a_n)_{n=0}^\infty$ such that
$$
\|(a_n)\|_{Z_{\mathrm{RUD}}}=\sup\{\langle \phi,\sum_n a_n u_n\rangle:\phi\in\mathcal{N}\}<\infty.
$$

\teor\label{MR-RUD}
The unit basis $(u_n)$ is a RUD basis of the space $Z_{RUD}$ endowed with the norm $\|\cdot\|_{Z_{RUD}}$
without unconditional subsequences. In addition, given any infinite set $N\subset\mathbb{N}$, if for every
$n\in \mathbb{N}$ we take $s_n\subset N$ such that $(s_1,\ldots, s_n)\in\mathcal{B}$, and let
$$
x_n=\frac1{\sqrt{\sharp s_n}}\sum_{j\in s_n}u_j,
$$
then $(x_n)$ is a normalized block sequence of $(u_n)_{n\in N}$ which is not RUD.
\fteor

\begin{proof}

Let us see first that $(u_n)_{n\in\mathbb{N}}$ is RUD. To this end, we take arbitrary scalars $(a_k)_{k=1}^l$
and let us prove that
$$
\Big\|\sum_{k=1}^l a_ku_k\Big\|=\sup_{\phi\in\mathcal{N}}\langle\phi,\sum_{k=1}^l a_ku_k\rangle\leq C \mathbb{E} \Big(\Big\|\sum_{k=1}^l \epsilon_k |a_k| e_k \Big\| \Big),
$$
for some constant $C$ independent on $(a_k)_{k=1}^l$. First, for $\phi=\pm u_k$ we clearly have
$$
\langle \phi,\sum_{k=1}^l a_ku_k\rangle\leq\max_k |a_k|\leq \mathbb{E} \Big(\Big\|\sum_{k=1}^l \epsilon_k |a_k| e_k \Big\| \Big).
$$
Now, suppose $\phi$ has the form
$$
\phi=\pm\sum_{i=1}^n\frac{1}{\sqrt{\sharp s_i}}\sum_{j\in s_i}u_j^*, \hspace{1cm}\textrm{ or }\hspace{1cm} \phi=\sum_{i=1}^n\frac{1}{\sqrt{\sharp s_i}}\sum_{j\in s_i}\varepsilon_i(j)u_j^*,
$$
for some fixed $(s_1,\ldots,s_n)\in\mathcal{B}$ and $\varepsilon_i\in\mathcal{E}(s_i)$. Let us consider the
set
$$
A=\{\varepsilon\in\{-1,1\}^l:\varepsilon|_{s_i}\in\mathcal{E}(s_i) \,\forall i=1,\ldots,n\}.
$$
Hence, for $(\theta_k)_{k=1}^l\in A$, we have that $\sum_{i=1}^n\frac{1}{\sqrt{\sharp s_i}}\sum_{j\in
s_i}\theta_j u_j^*\in\mathcal{N}$ so we get that, for both cases of $\phi$,
\begin{align*}
\Big\|\sum_{k=1}^l\theta_k|a_k|u_k\Big\|\geq &\langle\sum_{i=1}^n\frac{1}{\sqrt{\sharp s_i}}\sum_{j\in s_i}\theta_j u_j^*,
\sum_{k=1}^l\theta_k|a_k|u_k\rangle=  \sum_{i=1}^n\frac{1}{\sqrt{\sharp s_i}}\sum_{j\in s_i\cap\{1,\ldots,l\}}
|a_j|\geq \\
\geq & \langle\phi,\sum_{k=1}^l a_ku_k\rangle.
\end{align*}
In particular, we have
$$
\mathbb{E} \Big(\Big\|\sum_{k=1}^l \epsilon_k |a_k| e_k \Big\| \Big)\geq \mathbb{E} \Big(\Big\|\sum_{k=1}^l \epsilon_k |a_k| e_k \Big\|\chi_A((\epsilon_k)_{k=1}^l) \Big)\geq \langle\phi,\sum_{k=1}^l a_ku_k\rangle\frac{\sharp A}{2^l}.
$$
Notice that if we denote $m_{j_i}=\sharp s_i\in M$, then the cardinality of $A$ is given by
$$
\sharp A=\prod_{i=1}^n \sharp \mathcal{E}(s_i) \times 2^{l-\sharp(s_1\cup\ldots\cup s_n)}=2^l\prod_{i=1}^n\frac{k_{m_{j_i}}}{2^{m_{j_i}}}\geq2^l\prod_{i=1}^\infty\frac{k_{m_i}}{2^{m_i}}\geq 2^l (1-\delta),
$$
because of condition $(ii)$ in the definition of the sequence $M$. Thus, we have that
$$
\mathbb{E} \Big(\Big\|\sum_{k=1}^l \epsilon_k |a_k| e_k \Big\| \Big)\geq (1-\delta)\langle\phi,\sum_{k=1}^l a_ku_k\rangle.
$$

Therefore, we finally get that for any scalars $(a_k)_{k=1}^l$
$$
\mathbb{E} \Big(\Big\|\sum_{k=1}^l \epsilon_k a_k e_k \Big\| \Big)=\mathbb{E} \Big(\Big\|\sum_{k=1}^l \epsilon_k |a_k| e_k \Big\| \Big)\geq (1-\delta)\Big\|\sum_{k=1}^l a_ku_k\Big\|,
$$
so $(u_n)_{n\in\mathbb{N}}$ is RUD.

For the second part, given an infinite set $N\subset\mathbb{N}$, let
$$
x_n=\frac1{\sqrt{\sharp s_n}}\sum_{j\in s_n}u_j,
$$
where for every $n\in \mathbb{N}$, $s_n\subset N$ is such that $(s_1,\ldots, s_n)\in\mathcal{B}$. We claim
that for any scalars $(a_j)_{j=1}^n$ we have that
$$
\Big\|\sum_{j=1}^n a_j x_j\Big\|\approx\sup_{1\leq l\leq n}\Big|\sum_{k=1}^l a_k\Big|,
$$
independently of the scalars and $n\in\mathbb{N}$. In particular, by Theorem \ref{summing}, $(x_n)$ cannot be
RUD. Besides, since this holds for any $N\subset\mathbb{N}$, no subsequence of $(u_n)_{n\in\mathbb{N}}$ can
be unconditional.

First, for $l=1,\ldots,n$ let
$$
\phi_l=\sum_{i=1}^l\frac1{\sqrt{\sharp s_i}}\sum_{k\in s_i}u_k^*\in\mathcal{N}.
$$
Hence, we have that
$$
\Big\|\sum_{j=1}^n a_j x_j\Big\|\geq\sup_{1\leq l\leq n}\langle\pm\phi_l,\sum_{j=1}^n a_j x_j\rangle=\sup_{1\leq l\leq n}\pm\sum_{i=1}^l \frac{a_i}{\sharp s_i}\sharp s_i=\sup_{1\leq l\leq n}\Big|\sum_{i=1}^l a_i\Big|.
$$

For the converse inequality, first, if $\phi$ has the form $\phi=\pm u_k^*$ for $k\in s_i$, then we have that
$$
\langle \phi, \sum_{j=1}^n a_j x_j\rangle=\frac{a_i}{\sqrt{\sharp s_i}}\leq\sup_{1\leq i\leq n}|a_i|\leq2\sup_{1\leq l\leq n}\Big|\sum_{k=1}^l a_k\Big|.
$$
Now, suppose $\phi$ has the form
$$
\phi= \sum_{i=1}^m\frac{1}{\sqrt{\sharp t_i}}\sum_{l\in t_i}\varepsilon_i(l)u_l^*,
$$
for some $(t_1,\ldots,t_m)\in\mathcal{B}$ and $\varepsilon_i\in\mathcal{E}(t_i)$ for every $i=1,\ldots,m$, or
$\varepsilon_i=\varepsilon 1_{t_i}$ for every $i=1,\ldots,m$, with $\varepsilon\in\{-1,1\}$. Let
$$
j_0=\min \{j\leq n:s_j\neq t_j\}.
$$
Hence, we can write
$$
\langle \phi,\sum_{j=1}^n a_j x_j\rangle=\overbrace{\langle \phi,\sum_{j=1}^{j_0-1} a_j x_j\rangle}^{(A)}+\overbrace{\overset{\,}{\overset{\,}{\overset{\,}{\langle \phi, a_{j_0} x_{j_0}\rangle}}}}^{(B)}+\overbrace{\langle \phi,\sum_{j=j_0+1}^n a_j x_j\rangle}^{(C)}.
$$
Since for any $i\geq j_0>j$, we have $t_i\cap s_j=\emptyset$, and $s_k=t_k$ for $k<j_0$, we get that
$$
(A)=\langle\sum_{i=1}^{j_0-1}\frac{1}{\sqrt{\sharp t_i}}\sum_{l\in t_i}\varepsilon_i(l)u_l^*,\sum_{j=1}^{j_0-1} a_j x_j\rangle=\sum_{j=1}^{j_0-1} \frac{a_j}{\sharp s_j}\sum_{k\in s_j}\varepsilon_j(k).
$$
Thus, depending on the form of $\phi$ we either have $\sum_{k\in s_j}\varepsilon_j(k)=0$ for every
$j=1,\ldots, j_0-1$, or $\sum_{k\in s_j}\varepsilon_j(k)=\varepsilon\sharp s_j$ for every $j=1,\ldots,
j_0-1$, and some $\varepsilon\in\{-1,1\}$. In any case we get
$$
(A)\leq\Big|\sum_{j=1}^{j_0-1}a_j\Big|\leq\sup_{1\leq l\leq n}\Big|\sum_{i=1}^l a_i\Big|.
$$
Now, since for $i<j_0$, we have $t_i\cap s_{j_0}=\emptyset$, we get that
$$\begin{array}{lll}
(B)&=&\langle\sum_{i=j_0}^m\frac{1}{\sqrt{\sharp t_i}}\sum_{l\in t_i}\varepsilon_i(l)u_l^*,a_{j_0} \frac{1}{\sqrt{\sharp s_{j_0}}}\sum_{k\in s_{j_0}}u_k\rangle \\
&\leq &|a_{j_0}|\Big(\frac{1}{\sqrt{\sharp t_{j_0}\sharp s_{j_0}}}\sum_{l\in t_{j_0}\cap s_{j_0}}\varepsilon_{j_0}(l)+\sum_{i>j_0}\frac{1}{\sqrt{\sharp t_i\sharp s_{j_0}}}\sum_{l\in t_i\cap s_{j_0}}\varepsilon_{j_0}(l)\Big)\\
&\leq&|a_{j_0}|(1+\sum_j\sum_{k\neq j}\sqrt{\min\{\frac{m_j}{m_k},\frac{m_k}{m_j}\}})\leq(1+\delta)|a_{j_0}|.
\end{array}
$$
So we also get that
$$
(B)\leq 2(1+\delta)\sup_{1\leq l\leq n}\Big|\sum_{i=1}^l a_i\Big|.
$$
And, finally we have
$$
\begin{array}{lll}
(C)&=&\langle\sum_{i=j_0}^m\frac{1}{\sqrt{\sharp t_i}}\sum_{l\in t_i}\varepsilon_i(l)u_l^*,\sum_{j=j_0+1}^n a_j \frac1{\sqrt{\sharp s_j}}\sum_{k\in s_j}u_k\rangle\\
&=&\sum_{i=j_0}^m\sum_{j>j_0}\frac{a_j}{\sqrt{\sharp t_i \sharp s_j}}\sum_{l\in t_i\cap s_j}\varepsilon_i(l)\\
&\leq&\sup_{j_0<j\leq n}|a_j|\sum_{i=j_0}^m\sum_{j>j_0}\frac{\min\{\sharp t_i, \sharp s_j\}}{\sqrt{\sharp t_i \sharp s_j}}\\
&\leq&\sup_{j_0<j\leq n}|a_j|\sum_i\sum_{k\neq i}\sqrt{\min\{\frac{m_i}{m_k},\frac{m_k}{m_i}\}}\leq2\delta\sup_{1\leq l\leq n}\Big|\sum_{i=1}^l a_i\Big|.
\end{array}
$$

Thus, we have seen that for every $\phi\in\mathcal{N}$
$$
\langle \phi,\sum_{j=1}^n a_j x_j\rangle\leq(3+4\delta)\sup_{1\leq l\leq n}\Big|\sum_{i=1}^l a_i\Big|.
$$
Therefore, we get that
$$
\sup_{1\leq l\leq n}\Big|\sum_{i=1}^l a_i\Big|\leq\Big\|\sum_{j=1}^n a_j x_j\Big\|\leq(3+4\delta)\sup_{1\leq l\leq n}\Big|\sum_{i=1}^l a_i\Big|.
$$
\end{proof}


\begin{problem}
Does every Banach space have a RUC or RUD basic sequence?
\end{problem}

\section{RUD sequences in rearrangement invariant spaces}\label{Sec_RUD_Banach}

In the framework of Banach lattices, Krivine's functional calculus (cf. \cite[Section 1.d.]{LT2}) allows us to
give a meaning to expressions like $\big(\sum_{i=1}^n |x_i|^p\big)^{\frac1p}$, which coincides with the
corresponding pointwise operation when we deal with a Banach lattice of functions. Using Khintchine's
inequality we get a constant $C>0$ such that for any $(x_i)_{i=1}^n$ in an arbitrary Banach lattice $X$ we
have
$$
\int_0^1\Big\|\sum_{i=1}^n r_i(t)x_i\Big\|dt\geq\frac1C\Big\|\Big(\sum_{i=1}^n|x_i|^2\Big)^{\frac12}\Big\|.
$$
Moreover, if $X$ is $q$-concave for some $q<\infty$ (equivalently, if $X$ has finite cotype) then there is a
constant $C(q)>0$ such that a converse estimate holds:
$$
\int_0^1\Big\|\sum_{i=1}^n r_i(t)x_i\Big\|dt\leq C(q)\Big\|\Big(\sum_{i=1}^n|x_i|^2\Big)^{\frac12}\Big\|.
$$

In particular, a sequence $(x_n)_{n=1}^\infty$ in a Banach lattice $X$ with finite cotype is RUD if and only
if there is $K>0$ such that for any scalars $(a_k)_{k=1}^n$
$$
\Big\|\sum_{k=1}^n a_k x_k\Big\|\leq K \Big\|\Big(\sum_{k=1}^n |a_k x_k|^2\Big)^{\frac12}\Big\|.
$$

It is reasonable to expect that if the lattice structure has a lot of symmetry, then it is easier to find RUD
sequences. This is precisely stated in the next result for rearrangement invariant spaces which makes use of
the estimates for martingale difference sequences given in \cite{JS}.

\teor\label{ri RUD}
Let $X$ be a separable rearrangement invariant space on $[0,1]$ with non-trivial upper Boyd index. Every block
sequence of the Haar basis in $X$ is RUD. In particular, every weakly null sequence $(x_n)$ in $X$ has a
subsequence which is basic RUD.
\fteor

\begin{proof}
Let $(h_j)$ denote the Haar system on $[0,1]$. That is, for $j=2^k+l$, with $k\in\mathbb{N}$ and $1\leq l\leq
2^k$, we have
$$
h_j=\chi_{[\frac{2l-2}{2^{k+1}},\frac{2l-1}{2^{k+1}})}-\chi_{[\frac{2l-1}{2^{k+1}},\frac{2l}{2^{k+1}})}.
$$
By \cite[Proposition 2.c.1]{LT2}, $(h_j)$ is a monotone basis of $X$. Let us take a block sequence
$$
y_k=\sum_{j=p_k}^{q_k}b_j h_j
$$
(with $p_k\leq q_k<p_{k+1}$). Given scalars $(a_k)_{k=1}^m$ we can consider the sequence
$$
f_n=
\left\{
\begin{array}{cc}
\sum_{k=1}^n a_ky_k   &  n<m \\
&\\
\sum_{k=1}^m a_ky_k  &  n\geq m.
\end{array}
\right.
$$ It holds that $(f_n)$ is a martingale with respect to the filtration $(\mathcal{D}_{q_n})$, where $\mathcal{D}_{q_n}$ is the smallest $\sigma$-algebra $\mathcal{A}$ for which the functions $\{h_1,\ldots,h_{q_n}\}$ are $\mathcal{A}$-measurable.

By \cite[Theorem 3]{JS} there is $C>0$, which is independent of the scalars $(a_k)_{k=1}^m$, such that
$$
\Big\|\sum_{k=1}^m a_ky_k\Big\| \leq \|\sup_n |f_n|\| \leq C \Big\|\Big(\sum_{k=1}^\infty |f_k-f_{k-1}|^2\Big)^{\frac12}\Big\|= C \Big\|\Big(\sum_{k=1}^m |a_k y_k|^2\Big)^{\frac12}\Big\|.
$$
Now, by \cite[Theorem 1.d.6]{LT2} there is a universal constant $A>0$ such that
$$
\Big\|\Big(\sum_{k=1}^m |a_k y_k|^2\Big)^{\frac12}\Big\|\leq A\int_0^1\Big\|\sum_{k=1}^m r_k(s)a_ky_k\Big\|ds.
$$
Hence, since $(x_{n_k})$ is equivalent to $y_k$ we have that
$$
\Big\|\sum_{k=1}^m a_kx_{n_k}\Big\| \leq K \Big\|\sum_{k=1}^m a_ky_k\Big\| \leq CAK\int_0^1\Big\|\sum_{k=1}^m r_k(s)a_ky_k\Big\|ds \leq CAK^2\mathbb{E}\Big(\Big\|\sum_{k=1}^m \varepsilon_n a_k x_{n_k} \Big\| \Big).
$$
\end{proof}

A similar idea has been used in \cite{AKS} to show that if a separable r.i. space $X$ on $[0,1]$ is p-convex
for some $p>1$ and has strictly positive lower Boyd index, then $X$ has the Banach-Saks property.

\begin{corollary}\label{jms}
There is in $L_1$ a RUD basic sequence without unconditional subsequences.
\end{corollary}
\begin{proof}
Let $(f_n)_n$ be the weakly-null basic sequence in $L_1$ without unconditional subsequences given in
\cite{JMS}. Then any RUD subsequence of $(f_n)_n$ (existing by Theorem \ref{ri RUD}) fulfills the desired
requirements.
\end{proof}

Note that the Haar basis in $L_1[0,1]$ is a conditional basis such that every block is RUD (compare with
Theorem \ref{James_thm}). We do not know if a basis with the property that every block subsequence is RUD has some unconditional block subsequence. The sequence given in Corollary \ref{jms} satisfies that every block subsequence is RUD, and fails to have an unconditional subsequence (although it has unconditional blocks).

\section{average norms}
The motivating question here is the following: Given an unconditional basic sequence $(x_n)_n$, find an RUC or RUD
basis $(y_n)_n$ such that $(r_n \otimes y_n)_n$ is equivalent to $(x_n)_n$ but
\begin{enumerate}
\item $(y_n)_n$ is not equivalent to $(x_n)_n$, or
\item $(y_n)_n$ does not contain subsequences equivalent to subsequences of $(x_n)$, or
\item $(y_n)_n$ does not contain unconditional subsequences.
\end{enumerate}
\begin{problem}
Characterize unconditional sequences $(x_n)_n$ under one of the previous criteria.
\end{problem} In the case for the unit vector basis of $c_0$ or $\ell_1$ it is not possible to find such a basis as the
following well-known theorems show. By the sake of completeness, we will reproduce the original proofs.

\teor[S. Kwapien \cite{Kw}]\label{oiio3rere}
Suppose that $(x_n)_n$ is a seminormalized basic sequence in a Banach space such that $\sup_n \mathbb
E_{\vep}\nrm{\sum_{i=1}^n \vep_i x_i}<\infty$. Then $(x_n)_n$ has a subsequence equivalent to the unit basis
of $c_0$.
\fteor
\begin{proof}
For every measurable set $B\subseteq [0,1]$ one has that  $\lim_{n\to \infty}\la (\{t\in B\, : \,
r_n(t)=1\})=\lim_{n\to \infty}\la (\{t\in B\, : \,  r_n(t)=-1\})=(1/2)\la(B)$. Let $M>0$ be such that
$A=\conj{t\in [0,1]}{\sup_n \nrm{\sum_{i=1}^n r_i(t)x_i}\le M}$ has Lebesgue measure $\la(A)>1/2$. Now let
$n_1\in \N$ be such that
\begin{equation}
\la(\conj{t\in A}{r_{n_1}(t)=1})= \la(\conj{t\in A}{r_{n_1}(t)=-1})>\frac{1}{2^2}.
\end{equation}
In general, let $(n_k)_k$ be a strictly increasing sequence of integers such that for every $k$  and every
sequence of signs $(\vep_i)_{i=1}^k$ one has that
\begin{equation} \label{lkjdflksldkfsjdfe}
\la(A({(\vep_i)_{i=1}^k}))> \frac1{2^{k+1}}.
\end{equation}
where $A({(\vep_i)_{i=1}^k})= \conj{t\in A}{r_{n_i}(t)=\vep_i \text{ for every $1\le i\le k$}} $. Let now
$s_i=r_i$ if $i\in \{n_j\}_j$ and $s_i=-r_i$ if $i\notin \{n_j\}_j$.  Let
$$B=\conj{t\in [0,1]}{\sup_n\nrm{\sum_{i=1}^n s_i(t)x_i}\le M}.$$  Since $(r_i)_i$ and $(s_i)_i$ are
equidistributed, it follows that
\begin{equation}
\la(B((\vep_i)_{i=1}^k))=\la(A((\vep_i)_{i=1}^k))>\frac1{2^{k+1}},
\end{equation}
where $B((\vep_i)_{i=1}^k)=\conj{t\in B}{s_{n_i}=\vep_i \text{ for every $1\le i \le k$}}$. Since the set
$$
\bigcap_{i=1}^k \{r_{n_i}=\vep_i\}=\conj{t\in [0,1]}{r_{n_i}(t)=\vep_i\text{ for every $1\le i\le k$}}
$$
has measure $2^{-k}$, and since
$$
A((\vep_i)_{i=1}^k),B((\vep_i)_{i=1}^k)\subseteq \bigcap_{i=1}^k \{r_{n_i}=\vep_i\}
$$
it follows that
$$
A((\vep_i)_{i=1}^k)\cap B((\vep_i)_{i=1}^k) \neq \emptyset.
$$
Let $t_0\in A((\vep_i)_{i=1}^k)\cap B((\vep_i)_{i=1}^k) $. Hence,
\begin{equation}
\nrm{\sum_{i=1}^k\vep_i x_{n_i}}=\frac{1}{2}\nrm{\sum_{j=1}^{n_k} r_{j}(t_0) x_{j} +\sum_{j=1}^{n_k}s_j(t_0)x_j}\le M.
\end{equation}
Now it is easy to deduce from here that $\nrm{\sum_{i=1}^ka_i x_{n_i}}\le M \max_{i=1}^k |a_i|$.
\end{proof}

\prop[J. Bourgain \cite{Bo2}]\label{iuuiuiere}
Suppose that $(x_n)_n$ is a bounded sequence in a Banach space $X$ such that for some constant $\de>0$ one
has that
\begin{equation}\label{nnnkfkjgdf}
\text{$\mathbb E_\vep \nrm{\sum_{i=1}^n \vep_i a_i x_i}\ge \de\sum_{i=1}^n|a_i|$ for
every sequence of scalars $(a_i)_{i=1}^n$.}
\end{equation}Then $(x_n)_n$ has a subsequence equivalent to the unit basis of
$\ell_1$.
\fprop
\begin{proof}
By Rosenthal's $\ell_1$ theorem, we may assume otherwise that $(x_n)_n$ has a subsequence which is
weakly-Cauchy. Since our hypothesis \eqref{nnnkfkjgdf} passes to subsequences, we may assume without loss of
generality that $(x_n)_n$ is weakly-convergent to $x^{**}\in X^{**}$.  It is well known that for every
$\gamma>0$ there is a convex combination $(a_i)_{i=1}^n$ such that
\begin{enumerate}
\item[(a)]
 $\nrm{\sum_{i=1}^n a_i\vep_i x_i-(\sum_{i=1}^n \vep_i a_i)x^{**}}\le \gamma$  for every sequence of
 signs $(\vep_i)_{i=1}^n$.
 \item[(b)] $\nrm{(a_i)_{i=1}^n}_2\le \gamma$.
\end{enumerate}
Indeed, for the first part, think of each $x_n-x^{**}$ as a function in $C[0,1]$, and use Mazur's result for the
weakly-null sequence $(|x_n-x^{**}|)_n$; once (a) is established for each $\gamma$, let $n$ be such that
$\gamma\sqrt{n}\ge 1$, and find $s_1<\dots <s_n$ and convex combinations $\sum_{j\in s_i}a_ju_j$ fulfilling
(a) for $\gamma/n$; then the convex combination $(1/n)\sum_{i=1}^n \sum_{j\in s_i}a_j u_j$ satisfies that
$$
\nrm{(1/n)\sum_{i=1}^n\sum_{j\in s_i}\vep_j a_j (x_j-x^{**})}\le \gamma
$$
for every choice of signs $(\vep_i)_i$, and
$$
\nrm{(1/n)\sum_{i=1}^n \sum_{j\in s_i}a_j u_j}_2\le 1/\sqrt{n}\le\gamma.
$$

Now let $(a_i)_{i=1}^n$ be the corresponding combination for $\gamma$ such that $\gamma(1+\nrm{x^{**}})<\de$.
Then
\begin{equation}
\de \le \mathbb E_\vep \nrm{\sum_{i=1}^n \vep_i a_i x_i}\le \gamma + \nrm{x^{**}}\mathbb E_\vep |\sum_{i=1}^n a_i\vep_i| \le \gamma +
\nrm{x^{**}}(\sum_{i=1}^n a_i^2)^\frac12<\de,
\end{equation}
a contradiction.
\end{proof}

\eje
For each $1<p\le 2$, on  $c_{00}$ define the norm
$$\nrm{(a_i)_{i=1}^n}_{s,p}:=\max\{ \nrm{\sum_{i=1}^n a_i
s_i}_\infty,\nrm{(a_i)_{i=1}^n}_p\},$$
 where $(s_i)_i$ is the summing basis of $c_{0}$; let $X$
be the completion of $c_{00}$ under this norm. Then the unit Hamel basis $(u_n)_n$ is RUC and satisfies that
$(r_n \otimes u_n)_n$ is equivalent to the unit basis of $\ell_p$. It is easy to see that $(u_i)_{i}$ in $X$
does not have unconditional subsequences.
\feje

\end{document}